\documentclass[reqno]{amsart}
\usepackage{amssymb}
\usepackage{amsmath}
\usepackage{hyperref}
\hypersetup{ colorlinks = true, urlcolor = blue, linkcolor = blue, citecolor = red }
\usepackage{xcolor}
\usepackage{enumitem}
\usepackage{esint}
\makeatletter
\def\namedlabel#1#2{\begingroup
	#2%
	\def\@currentlabel{#2}%
	\phantomsection\label{#1}\endgroup
}
\usepackage{varwidth}
\usepackage{tasks}
\DeclareMathOperator{\dv}{div}
\DeclareMathOperator{\loc}{loc}

\newcommand{\RR}{\mathbb{R}}
\newcommand{\mA}{\mathcal{A}}
\newcommand{\Om}{\Omega}
\newcommand{\na}{\nabla}
\newcommand{\pa}{\partial}
\newcommand{\La}{\Lambda}
\newcommand{\la}{\lambda}
\newcommand{\ep}{\epsilon}
\newcommand{\om}{\omega}
\newcommand{\lao}{\lambda_\om}
\newcommand{\data}{\mathit{data}}
\usepackage[T1]{fontenc}
\usepackage[utf8]{inputenc}


\theoremstyle{plain}
\newtheorem{theorem}{Theorem}[section]
\newtheorem{lemma}[theorem]{Lemma}
\newtheorem{corollary}[theorem]{Corollary}

\newtheorem{definition}[theorem]{Definition}

\newtheorem{assumption}[theorem]{Assumption}
\newtheorem{remark}[theorem]{Remark}
\def\Xint#1{\mathchoice
	{\XXint\displaystyle\textstyle{#1}}%
	{\XXint\textstyle\scriptstyle{#1}}%
	{\XXint\scriptstyle\scriptscriptstyle{#1}}%
	{\XXint\scriptstyle\scriptscriptstyle{#1}}%
	\!\int}
\def\XXint#1#2#3{{\setbox0=\hbox{$#1{#2#3}{\int}$}
		\vcenter{\hbox{$#2#3$}}\kern-.5\wd0}}

\def\Yint#1{\mathchoice
	{\YYint\displaystyle\textstyle{#1}}%
	{\YYint\textstyle\scriptstyle{#1}}%
	{\YYint\scriptstyle\scriptscriptstyle{#1}}%
	{\YYint\scriptscriptstyle\scriptscriptstyle{#1}}%
	\!\iint}
\def\YYint#1#2#3{{\setbox0=\hbox{$#1{#2#3}{\iint}$}
		\vcenter{\hbox{$#2#3$}}\kern-.51\wd0}}
\def\longdash{{-}\mkern-3.5mu{-}} 
\def\fiint{\Yint\longdash}

\def\Xint#1{\mathchoice
	{\XXint\displaystyle\textstyle{#1}}%
	{\XXint\textstyle\scriptstyle{#1}}%
	{\XXint\scriptstyle\scriptscriptstyle{#1}}%
	{\XXint\scriptscriptstyle\scriptscriptstyle{#1}}%
	\!\int}
\def\XXint#1#2#3{{\setbox0=\hbox{$#1{#2#3}{\int}$ }
		\vcenter{\hbox{$#2#3$ }}\kern-.6\wd0}}

\def\dashint{\Xint-}

\DeclareMathOperator{\diam}{diam}

\usepackage{nameref}
\makeatletter
\let\orgdescriptionlabel\descriptionlabel
\renewcommand*{\descriptionlabel}[1]{%
	\let\orglabel\label
	\let\label\@gobble
	\phantomsection
	\edef\@currentlabel{#1}%
	\let\label\orglabel
	\orgdescriptionlabel{#1}%
}
\makeatother
\numberwithin{equation}{section}
\def\Xint#1{\mathchoice
    {\XXint\displaystyle\textstyle{#1}}%
    {\XXint\textstyle\scriptstyle{#1}}%
    {\XXint\scriptstyle\scriptscriptstyle{#1}}%
    {\XXint\scriptscriptstyle\scriptscriptstyle{#1}}%
    \!\int}
\def\XXint#1#2#3{\setbox0=\hbox{$#1{#2#3}{\int}$}
    \vcenter{\hbox{$#2#3$}}\kern-0.5\wd0}
\def\fint{\Xint-}
\def\dashint{\Xint{\raise4pt\hbox to7pt{\hrulefill}}}

\def\XXiint#1#2#3{\setbox0=\hbox{$#1{#2#3}{\iint}$}
    \vcenter{\hbox{$#2#3$}}\kern-0.5\wd0}

\allowdisplaybreaks

\begin{document}
	
\title[Estimate for the singular case]{Calder\'on-Zygmund type estimate for the singular parabolic double-phase system}

\everymath{\displaystyle}

\makeatletter
\@namedef{subjclassname@2020}{\textup{2020} Mathematics Subject Classification}
\makeatother

 \author{Wontae Kim}
 \address[Wontae Kim]{Department of Mathematics, Uppsala University, P.O. BOX 480, 751 06 Uppsala, Sweden}
 \email{kim.wontae.pde@gmail.com}
%\thanks{???}

\begin{abstract}
    This paper discusses the local Calder\'on-Zygmund type estimate for the singular parabolic double-phase system. The proof covers the counterpart $p<2$ of the result in \cite{K2024}.  Phase analysis is employed to determine an appropriate intrinsic geometry for each phase. Comparison estimates and scaling invariant properties for each intrinsic geometry are the main techniques to obtain the main estimate.
\end{abstract}

\maketitle

\section{Introduction}
We study the gradient estimate for the parabolic double-phase system
\[
    u_t-\dv(b(z)(|\na u|^{p-2}\na u+a(z)|\na u|^{q-2}\na u))=-\dv(|F|^{p-2}F+a(z)|F|^{q-2}F)
\]
in $\Om_T=\Om\times (0,T)$ where $\Omega$ is a bounded domain in $\mathbb{R}^n$, $n\geq 2$, $T>0$ and the coefficient function $b(z)$ satisfies the ellipticity condition in \eqref{ellipticity}. Throughout the paper, we shall assume that the coefficient function $a(z)$ is non-negative and $(\alpha,\alpha/2)$-H\"older continuous for $\alpha\in(0,1]$, that is, there exists a constant $[a]_\alpha>0$ such that
\begin{equation}\label{coeff}
    |a(x,t)-a(y,t)|\le [a]_{\alpha}|x-y|^\alpha,\quad
    |a(x,t)-a(x,s)|\le [a]_{\alpha}|t-s|^\frac{\alpha}{2}
\end{equation}
for all $x,y\in \Om$ and $t,s\in (0,T)$ while exponents $p$ and $q$ satisfy
\begin{equation} \label{range_pq}
\frac{2n}{n+2}< p \leq 2, \quad p<q\le p+\frac{\alpha(p(n+2)-2n)}{2(n+2)}.    
\end{equation}
Note that $\tfrac{\alpha(p(n+2)-2n)}{2(n+2)}=\tfrac{\alpha p}{n+2}\tfrac{p(n+2)-2n}{2p}$
where $\tfrac{p(n+2)-2n}{2p}$
is the scaling deficit of the singular $p$-Laplace system as in \cite{MR1230384}. 
The aim of this paper is to prove the Calder\'on-Zygmund type estimate of the following implication
\begin{equation}\label{implication}
    |F|^p+a|F|  ^q \in L^\sigma_{\loc}\Longrightarrow |\na u|^p+a|\na u|^q\in L^\sigma_{\loc}
\end{equation}
for all $\sigma\in(1,\infty)$.

The double-phase system has a non-standard growth condition due to the presence of the coefficient $a(z)$. For each point $z$, if $a(z)=0$, the system is reduced to the $p$-Laplace system while, if $a(z)\ne0$, the system is the $(p,q)$-Laplace system. It is presumed that double-phase systems exhibit two different phases, nevertheless, further analysis is necessary as $a(z)\ne0$ does not always imply $a(\cdot)$ is comparable in the neighborhood of $z$. For such a neighborhood, arguments in the $(p,q)$-Laplace system cannot be utilized. Moreover, as nonlinear parabolic systems demand intrinsic geometries for the regularity theory, it is necessary to connect phase and intrinsic geometry. In this paper, we adopt the phase analysis for the double-phase system developed in \cite{KKM} to provide the proper intrinsic geometry for each point. In our phase analysis, there are two types of phase, $p$-intrinsic case and $(p,q)$-intrinsic case. In the $p$-intrinsic case, estimates for the double-phase system are treated in the $p$-intrinsic geometry, which is intrinsic geometry for the $p$-Laplace system. Despite there being a $q$-Laplace part $a|\na u|^{p-2}\na u$, those terms from $q$-Laplace part are perturbed to terms from the $p$-Laplace part $|\na u|^{p-2}\na u$. Furthermore, we will see that in this case, the double phase system is scaling invariant under the $p$-intrinsic geometry. In contrast, if  $(p,q)$-intrinsic case holds, then we will show that there exists a neighborhood in which $a(\cdot)$ is comparable and we will apply the intrinsic geometry of the $(p,q)$-Laplace system. 

Additionally, we point out that the existence of the upper bound for $q$ in \eqref{range_pq} naturally arises in the non-standard growth problems. The term $\tfrac{\alpha p}{n+2}\tfrac{p(n+2)-2n}{2p}$ in the upper bound appears to be natural, but unlike in elliptic double phase system in \cite{MR2076158,MR2058167}, sharpness for \eqref{range_pq} is not known to the best of our knowledge.

The regularity properties of non-standard growth problems were first studied for elliptic equations in \cite{MR969900,MR1094446}. The development of regularity results for elliptic double-phase problems and its phase analysis are proved \cite{baasandorj2023self,MR3348922,MR3360738,MR3294408,de2023regularity,de2024regularity,MR2076158}. For the parabolic case, non-standard problems have been addressed in \cite{MR3073153,MR3356846}, while regularity results for the parabolic double-phase problem can be found in \cite{KKM,KKS,kim2024holder,MR4718687,MR3532237}. We also refer to \cite{cupini2024regularity,marcellini2023local} for more general structures of non-standard growth problems.

Regarding Calderón-Zygmund estimates, the elliptic $p$-Laplace system has been studied extensively, with key results in \cite{MR3035434,MR2187159,MR2345911,MR1486629,MR1246185,MR722254,MR1720770}, while
the parabolic 
$p$-Laplace system was established in \cite{MR2286632}. The elliptic double-phase system case has been considered in \cite{MR3447716,MR3985927}.
For the parabolic double-phase system, the degenerate case $(p\ge2)$ was established in \cite{K2024}. This paper extends the analysis to cover the singular case $(p<2)$.

%%%%%%%%%%%%%%%%%%%%%%%%%%%%%%%%%%%%%%%%%%%%%%%%%%%%%%%%%%%%%%%%%%%%%%%%%%%%%%%%%%%%%%%%%%%%%%%%%%%%%%%%%%%%%%%%%%%%%%%%%%%%%%%%%%%%%%%%%%%%%%%%%%%%%%%%%%%%%%%%%%%%%%%%%%%%%%%%%%%%%%%%%%%%%%%%%%%%%%%%%%%%%%%%%%%%%%%%%%%%%%%%%%%%%%%%%%%%%%%%%%%%%%%%%%%%%%%%%%%%%%%%%%%%%%%%%%%%%%%%%%%%%%%%%%%%%%%%%%%%%%%%%%%%%%%%%%%%%%%%%%%%%%%%%%%%%%%%%%%%%%%%%%%%%%%%%%%%
\section{Notations and main result}

\subsection{Notations}
For a point $z\in\RR^{n+1}$, we denote $z=(x,t)$ where $x\in \RR^n$ and $t\in \RR$.
A ball with centered at $x_0\in\RR^n$ and radius $\rho>0$ is denoted as
\[
    B_\rho(x_0)=\{x\in \RR^n:|x-x_0|<\rho\}.
\]
Parabolic cylinder centered at $z_0=(x_0,t_0)$ and its time interval are denoted as
\[
    Q_\rho(z_0)=B_\rho(x_0)\times I_\rho(t_0),\quad I_\rho(t_0)=(t_0-\rho^2,t_0+\rho^2).
\]
For $a(z)$ described in \eqref{coeff}, we define a functional $H(z,s):\Omega_T\times \RR^+\mapsto\RR^+$ as
\[
    H(z,s)=s^p+a(z)s^q.
\]
In this paper, we use two types of intrinsic cylinders. For $\la\geq1$ and $\rho > 0$, a $p$-intrinsic cylinder centered at $z_0=(x_0,t_0)$ is 
\begin{equation}\label{def_p_cylinder}
	Q_\rho^\la(z_0)= B^\la_\rho(x_0)\times I_{\rho}(t_0), \quad B^\la_\rho(x_0) = B_{\la^\frac{p-2}{2}\rho}(x_0), 
\end{equation}
and a $(p,q)$-intrinsic cylinders centered at $z_0=(x_0,t_0)$ is
\begin{align}\label{def_G_cylinder}
G_{\rho}^\la(z_0)=B^\la_{\rho}(x_0)\times J_{\rho}^\la(t_0),\quad J_\rho^{\la}(t_0)= \Bigl(t_0-\tfrac{\la^p}{H(z_0,\la)}\rho^2,t_0+\tfrac{\la^p}{H(z_0,\la)}\rho^2 \Bigr).  
\end{align}
The time interval includes the information of $z_0$, however, we always omit $z_0$ for the time interval as $H(z_0,\la)$ will remain fixed during our proof.
Nevertheless $G_\rho^\la(z_0)$ has the scaling factor $\la$ both in space and time direction, note that $\tfrac{\la^p}{H(z_0,\la)}\rho^2=\tfrac{\la^2}{H(z_0,\la)}(\la^\frac{p-2}{2}\rho)^2$ and thus $G_\rho^\la(z_0)$ is the standard intrinsic cylinder for $(p,q)$-Laplace system.
For $d>0$, we write
\[
    dQ_\rho^\la(z_0)=Q_{d\rho}^\la(z_0),
    \quad 
    dG_\rho^\la(z_0)=G_{d\rho}^\la(z_0).
\]

Finally, for $f\in L^1(\Omega_T,\RR^N)$ and a measurable set $E\subset\Om_T$ with $0<|E|<\infty$, we denote the integral average of $f$ over $E$ as
\[
	(f)_{E}=\frac{1}{|E|}\iint_{E}f\,dz=\fiint_{E}f\,dz.
\]

\subsection{Main result}
This paper is concerned with the parabolic double-phase system
\begin{equation}\label{eq}
    u_t-\dv\left(b(z)\mA(z,\na u)\right)=-\dv \mA(z,F)\quad\text{in}\quad \Om_T,
\end{equation}
where we abbreviate the parabolic double-phase operator as
\[\mA(z,\xi)=|\xi|^{p-2}\xi+a(z)|\xi|^{q-2}\xi\]
for $z\in\Om_T$ and $\xi\in\RR^{Nn}$ with $N\ge1$
and $b(z)$ is a positive measurable function satisfying the ellipticity condition
\begin{equation}\label{ellipticity}
    0<\nu\le b(z)\le L<\infty\quad\text{for a.e.}\quad z\in \Om_T.
\end{equation}
The weak solution to \eqref{eq} is defined in the following sense.
\begin{definition}
    A measurable map $u: \Om_T\mapsto\RR^N$ such that
    \begin{align*}
    \begin{split}
        &u\in C(0,T;L^2(  \Omega  ,\RR^N))\cap L^1(0,T;W_0^{1,1}(\Om,\RR^N))\\
        &\qquad\text{with}\quad\iint_{ \Om_T}H(z,|u|)+H(z,|\na u|)\,dz<\infty
    \end{split}
    \end{align*}
    is a weak solution to \eqref{eq} if for every $\varphi\in C_0^\infty(   \Omega_T  ,\RR^N)$, there holds
    \[ \iint_{ \Om_T}\left(-u\cdot \varphi_t+b(z)\mA(z,\na u)\cdot \na\varphi\right)\,dz=\iint_{\Om_T}\mA(z,F)\cdot \na\varphi\,dz.\]
\end{definition}

Some estimates of weak solutions to \eqref{eq} involve data of $u$ and $F$. For this, we write $c=c(\data)$ if the constant $c$ depends on the following values
  \begin{align*}
  \begin{split}
      n,N,p,q,\alpha,\nu,L,[a]_{\alpha},\diam(\Omega),\|u\|_{L^\infty(0,T;L^2(\Omega))},\|H(z,|\na u|)+H(z,|F|)\|_{L^1(\Omega_T)}.
  \end{split}
\end{align*}

Before we introduce the main result of this paper, we first state the partial result. In fact, it will play a crucial part in proving the main result.
\begin{theorem}[\cite{MR4718687}, Higher integrability]\label{higher}
    Let $u$ be a weak solution to \eqref{eq}. Then there exist $\varepsilon_0=\varepsilon_0(\data)\in (0,1)$ and $c=c(\data,\|a\|_{L^\infty(\Om_T)})$ such that for any $Q_{2\rho}(z_0)\subset  \Om_T$ and $\varepsilon\in(0,\varepsilon_0]$, there holds
    \begin{align*}
\begin{split}
    \fiint_{Q_{\rho}(z_0)}(H(z,|\na u|))^{1+\varepsilon}\,dz
        &\le c\left(\fiint_{Q_{2\rho}(z_0)}H(z,|\na u|)\,dz\right)^{1+\frac{2q\varepsilon }{p(n+2)-2n}}\\
        &\qquad+c\left(\fiint_{Q_{2\rho}(z_0)}(H(z,|F|))^{1+\varepsilon}\,dz+1\right)^{\frac{2q}{p(n+2)-2n}}.
\end{split}
\end{align*}
\end{theorem}

To prove the full range $\sigma$ in \eqref{implication}, we further assume the following two conditions. Firstly, we assume the coefficient $b$ has the VMO condition
\begin{equation}\label{vmo}
    \lim_{r\to0^+}\sup_{|I|\le 2r^2}\sup_{B_r(x_0)\subset \Om}\fiint_{B_r(x_0)\times I}|b(z)-(b)_{ B_{r}(x_0)\times I}|\,dz=0,
\end{equation}
where $I\subset (0,T)$ is any open interval.
Secondly, we will assume 
\begin{equation}\label{inf_a}
    \inf_{z\in\Om_T}a(z)>0.
\end{equation}
With these assumptions, the Calder\'on-Zygmund type estimate is as follows.
\begin{theorem}\label{main_theorem}
    Let $u$ be a weak solution to \eqref{eq} with assumptions \eqref{vmo} and \eqref{inf_a}. Suppose $Q_{4R}(z_0)\subset \Om_T$ for some $R\in(0,1)$.
    Then there exists $\rho_0\in(0,R)$ depending on 
    \[\data,\|H(z,|F|)\|_{L^{1+\varepsilon_0}(\Om_T)},\|a\|_{L^\infty(\Om_T)},R\]
    such that for any $\sigma\in(1+\varepsilon_0,\infty)$ and $\rho\in(0,\rho_0)$, there holds
    \begin{align*}
\begin{split}
    \fiint_{Q_{\rho}(z_0)}(H(z,|\na u|))^{\sigma}\,dz
        &\le c\left(\fiint_{Q_{2\rho}(z_0)}H(z,|\na u|)\,dz\right)^{1+\frac{2q(\sigma-1) }{p(n+2)-2n}}\\
        &\qquad+c\left(\fiint_{Q_{2\rho}(z_0)}(H(z,|F|))^{\sigma}\,dz+1\right)^{\frac{2q}{p(n+2)-2n}},
\end{split}
\end{align*}
where $c=c(\data,\|a\|_{L^\infty(\Om_T)},\sigma)$.
\end{theorem}

\begin{remark}
    We point out that the assumption \eqref{inf_a} is made purely for technical reasons and does not diminish the novelty of our paper.
   It might be misconstrued that Theorem~\ref{main_theorem} could be deduced from the estimate of the  $(p,q)$-Laplace system where $a$ is constant. If \eqref{eq} is interpreted as a $(p,q)$-Laplace system, then $\inf a$ serves as the lower bound for the ellipticity constant, resulting in the constant in the estimate depending on $\inf a$ and diverging as $\inf a$ approaches $0^+$. Indeed, regarding $c|\na u|^{q-2}\na u$ as a $q$-Laplace part with fixed constant $c>0$ locally, the remaining term $c^{-1}a(z)$ is considered as the coefficient function to proceed further by adopting technique in $(p,q)$-Laplace system. However, as presented, our estimate remains stable with respect to $\inf a$.

   In this paper, the assumption \eqref{inf_a} is employed only to construct the Dirichlet boundary problem, as there is no existence result when $\inf a =0$.
    This assumption characterizes the double-phase operator as a $q$-Laplace type given as
    \[
    \inf_{z\in \Om_T}a(z) |\xi|^q\le\mathcal{A}(z,\xi)\cdot\xi\le (1+\|a\|_{L^\infty(\Om_T)})(1+|\xi|)^q
    \]
    and the existence result of the $q$-Laplace type system can be employed.
    Moreover, as noted in \cite{KKS}, the existence of the Dirichlet boundary problem when $\inf a=0$ can be proved by applying the global Calderón-Zygmund type estimate.
    \end{remark}

\section{Comparison estimates}\label{com_sec}
This section aims to provide comparison estimates. As the double-phase system \eqref{eq} has two distinct phases, it is necessary to establish these estimates for each phase. We will explain the heuristic approach for distinguishing between the phases and provide a more detailed description in the next section.

In the Calder\'on-Zygmund type estimate of the double-phase system, we consider the upper-level set
\[
U=\{H(z,|\na u(z)|)>\La\}
\]
for each sufficiently larger $\La>1+\|a\|_{L^\infty(\Om_T)}$. In order to study the intrinsic geometry, for each $\omega\in U$, we defined $\la_\omega$ to be
\[
\La=H(\omega,\la_\omega)=\la_\omega^p+a(\omega)\la_\omega^q.
\]
Since $H(\om,|s|)$ is   an increasing  function on $|s|$, it easily follows that
\[
|\na u(\om)|>\la_\om.
\]
For the constant $K$ defined as
\begin{align}\label{K_def}
   180(1+[a]_{\alpha})\left(\frac{1}{|B_1|}\iint_{Q_{2\rho_0}(z_0)}\left(H(z,|\na u|)+\delta^{-1}H(z,|F|)\right)\,dz+1\right)^\frac{\alpha}{n+2},
\end{align}
where constant $\delta\in(0,1)$, either of the following holds
\[
K^2\la_\om^p\ge a(\om)\la_\om^q\quad \text{or}\quad K^2\la_\om^p< a(\om)\la_\om^q.
\]
The first case is equivalent to $a(\om)\le K^2\la_\om^{p-q}$ and it changes terms deduced from the $q$-Laplace part, $a(z)|\na u|^{q-2}\na u$, into the term of the $p$-Laplace part on some neighborhood of $\om$ in the context of intrinsic geometry. Moreover, this condition enforces the $q$-Laplace part invariant under
the scaling argument in the $p$-intrinsic geometry \eqref{def_p_cylinder}, see Lemma~\ref{p_sc_lem}. On the other hand, if the second case holds, then we will prove $a(z)$ is comparable on some neighborhood of $\om$ and $(p,q)$-intrinsic geometry in \eqref{def_G_cylinder} would be applied for the discussion.

In this section, constants $\ep,\delta,\rho_0$ will be used throughout the paper to carry out comparison estimates and the estimate in Theorem~\ref{main_theorem}. The constant $\epsilon\in (0,1)$ will be used for the iteration argument and be determined later in \eqref{def_ep}. The constant $\delta\in (0,1)$, which also affects $K$ in \eqref{K_def}, will be utilized to derive comparison estimates and be chosen depending on $\ep$ and $\data$. Finally, $\rho_0\in(0,1)$ will also be used for obtaining comparison estimates, be selected after taking $\delta$ and depend on $\epsilon$, $\delta$, $\data$, $\|a\|_{L^\infty(\Om_T)}$ and $\| H(z,|F|)\|_{L^{1+\varepsilon_0}(\Om_T)}$. On the other side, we will encounter the situation that constants in some estimates will also depend on $\delta$. For this case, we will write
\[
c_\delta=c(...,\delta).
\]
Finally, we shorten the following constant
\begin{align}\label{V_def}
    V=9K.
\end{align}
This constant will be used for the Vitali covering constant of our case in Lemma~\ref{Vitali}. 

\subsection{$p$-intrinsic case.}
In this subsection, we will obtain comparison estimates for the case $K^2\la_\om^p\ge a(\om)\la_\om^q$ with the assumptions on the stopping time argument in the $p$-intrinsic cylinder defined as in \eqref{def_p_cylinder}. 
\begin{assumption}\label{p_ass}
   For $\om=(y,s)\in Q_{R}(z_0)$, there exist $\lao>1$ and $\rho_\om\in(0,\rho_0)$ such that $Q_{16V\rho_\om}^{\la_\om}(\om)\subset Q_{2R}(z_0)$ and satisfying the following conditions.
    \begin{enumerate}[label={(\roman*)}]
        \item\label{i} $p$-intrinsic case: $K^2\lao^p\ge a(\om)\lao^q$,
    \item stopping time argument for $p$-intrinsic cylinder: 
    \begin{enumerate}[label=(\alph*),series=theoremconditions]
        \item\label{a} $\fiint_{Q_{16V\rho_\om}^{\lao}(\om)}\left(H(z,|\na u|)+\delta^{-1}H(z,|F|)\right)\,dz< \lao^p$,
        \item\label{b} $\fiint_{Q_{\rho_\om}^{\lao}(\om)}\left(H(z,|\na u|)+\delta^{-1}H(z,|F|)\right)\,dz=\lao^p$,
\end{enumerate}
\end{enumerate}
\end{assumption}
In this subsection, we omit the referenced point $\om$ and write $\lao$, $\rho_\om$ $Q_{\rho_\om}^{\lao}(\om)$ as $\la$, $\rho$ and $Q_{\rho}^\la$ for simplicity.

Along with the stopping time argument assumption, the following energy bounds hold.
\begin{lemma}\label{en_lem}
    There exists $c_\delta=c(\data,\delta)$ such that
    \[
    \sup_{I_{8V\rho}}\fint_{B_{8V\rho}^\la}\frac{|u-u_0|^2}{(8V\rho )^2}\,dx+\fiint_{Q_{8V\rho}^{\la}}\frac{|u-u_0|^p}{(8V\la^{\frac{p-2}{2}}\rho )^p}\,dz<c_\delta\la^p,
    \]
    where we shorten the notation
    \[
    u_0=(u)_{Q_{8V\rho}^\la}=(u)_{Q_{8V\rho_\om}^{\lao} (\om)}.
    \]
\end{lemma}
\begin{proof}
    The proof of this estimate is based on the Caccioppoli inequality and uses \ref{i} and \ref{a} for the conclusion. In particular, note that \ref{a} implies
    \[
    \fiint_{Q_{16V\rho}^{\lao}}\left(H(z,|\na u|)+H(z,|F|)\right)\,dz< \la^p
    \]
    The conclusion follows from the argument in \cite[Lemma~3.6 and $(3.8)$]{MR4718687} by replacing $K$ in there with \eqref{K_def}.
\end{proof}

\begin{remark}
    The parabolic Poincare inequality with the previous lemma leads to
    \[
    \fiint_{Q_{V\rho}^\la} \frac{|u-u_0|^\vartheta}{(8V\la^{\frac{p-2}{2}}\rho )^\vartheta}\,dz \le c_\delta \la^\vartheta
    \]
    for any $\vartheta\in [1,\tfrac{p(n+2)}{n}]$ where $c_\delta=c(\data,\delta)$.
\end{remark}

The above inequality is first established for the $p$-Laplace problems in \cite{MR1749438}. The $p$-intrinsic geometry in \eqref{def_p_cylinder} plays a role in assigning the same $\vartheta$ to both sides of the inequality. Meanwhile, for the double-phase problem, it is necessary to perturb the term, produced by the $q$-Laplace part like
\[
\rho^\alpha\fiint_{Q_{V\rho}^\la} \frac{|u-u_0|^\vartheta}{(8V\la^{\frac{p-2}{2}}\rho )^\vartheta}\,dz,
\]
into terms from the $p$-Laplace part. Moreover, it is relevant to the admissible range of $q$.
We put this issue in the intrinsic geometry setting in the following lemma.

\begin{lemma}\label{p_lem}
   For any constant $1<c_\delta=c(\data,\|a\|_{\infty},\|H(z,|F|)\|_{L^{1+\varepsilon_0}(\Om_T)},\delta)$, there exists $\rho_0=\rho_0(\data,\|a\|_{\infty},\|H(z,|F|)\|_{L^{1+\varepsilon_0}(\Om_T)},R,\delta,\ep)\in(0,1)$ such that if $\rho\in (0,\rho_0)$, then
\[c_\delta\rho^\alpha\la^q\le \frac{1}{(2V)^{n+2}2^{2q}3}\epsilon\la^p.\]
 \end{lemma}

 \begin{proof}
   Since it is assumed $Q_{4R}(z_0)\subset \Om_T$, we apply Theorem~\ref{higher} to obtain
   \begin{align*}
       \begin{split}
    \fiint_{Q_{2R}(z_0)}(H(z,|\na u|))^{1+\varepsilon_0}\,dz
        &\le c\left(\fiint_{Q_{4R}(z_0)}H(z,|\na u|)\,dz\right)^{1+\frac{2q\varepsilon_0 }{p(n+2)-2n}}\\
        &\qquad+c\left(\fiint_{Q_{4R}(z_0)}(H(z,|F|))^{1+\varepsilon_0}\,dz+1\right)^{\frac{2q}{p(n+2)-2n}},
\end{split}
   \end{align*}
   where $\varepsilon_0=\varepsilon_0(\data)$ and $c=c(\data,\|a\|_{L^\infty(\Om_T)})$. Therefore we have
   \[
   \iint_{Q_{2R}(z_0)}(H(z,|\na u|))^{1+\varepsilon_0}\,dz\le c_R,
   \]
   where $c_R=c_R(\data,\| a\|_{L^\infty(\Om_T)}, \|H(z,|F|)\|_{L^{1+\varepsilon_0}(\Om_T) },R)$. On the other side, we deduce from \ref{b} and $Q_{\rho}^\la\subset Q_{2R}(z_0)$ that
   \begin{align*}
       \begin{split}
           \la^p
           &=\fiint_{Q_{\rho}^{\la}}\left(H(z,|\na u|)+\delta^{-1}H(z,|F|)\right)\,dz\\
           &\le \left(\fiint_{Q_{\rho}^{\la}}\left(H(z,|\na u|)+\delta^{-1}H(z,|F|)\right)^{1+\varepsilon_0}\,dz\right)^\frac{1}{1+\varepsilon_0}\\
           &\le c_R|Q_{\rho}^{\la}|^{-\frac{1}{1+\varepsilon_0}}\\
           &\le c_R \Bigl(\la^{\frac{n(p-2)}{2}}\rho^{n+2}\Bigr)^{-\frac{1}{1+\varepsilon_0}}.
       \end{split}
   \end{align*}
Thus we get
\[
\la^\frac{\alpha p}{n+2}=(\la^p)^\frac{\alpha}{n+2}\le c_R\Bigl(\la^{\frac{n(p-2)}{2}}\rho^{n+2}\Bigr)^{-\frac{\alpha}{(1+\varepsilon_0)(n+2)}}.
\]   
In order to reach the conclusion, we use the above inequality to get
\begin{align*}
    \begin{split}
        c_\delta\rho^\alpha\la^q
        &=c_\delta\rho^\alpha\la^{q-\frac{\alpha p}{n+2}} \la^\frac{\alpha p}{n+2}\\
        &\le c_\delta c_R\rho^{\frac{\alpha\varepsilon_0}{1+\varepsilon_0}}\la^{q-\frac{\alpha p}{n+2} +\frac{\alpha n(2-p)}{2(1+\varepsilon_0)(n+2)} }. 
    \end{split}
\end{align*}
Since it follows from \eqref{range_pq} that
\[
q-\frac{\alpha p}{n+2}+\frac{\alpha n (2-p)}{2(n+2)}=q-\frac{\alpha(p(n+2)-2n)}{2(n+2)}\le p,
\]
we have
\[
q-\frac{\alpha p}{n+2}+\frac{\alpha n (2-p)}{2(n+2)(1+\varepsilon_0)}\le p
\]
and thus 
\[
c_\delta\rho^\alpha\la^q\le c_\delta c_R \rho_0^{\frac{\alpha \varepsilon_0}{1+\varepsilon_0}} \la^p.
\]
The proof is completed if we take $\rho_0$ sufficiently small.
\end{proof}

We now start to construct maps to apply comparison estimates.
Consider the weak solution  
\[
\zeta\in C(I_{8V\rho};L^2(B_{8V\rho}^\la,\RR^N))\cap L^q(I_{8V\rho};W^{1,q}(B_{8V\rho}^\la,\RR^N))
\]
to the Dirichlet boundary problem
\begin{align*}
    \begin{cases}
        \zeta_t-\dv (b(z)\mathcal{A}(z,\na \zeta))=0&\text{in}\quad Q_{8V}^\la,\\
        \zeta=u-u_0&\text{on}\quad \pa_pQ_{8V\rho}^\la.
    \end{cases}
\end{align*}

\begin{lemma}\label{p_c1}
    There exist $\delta=\delta(\data,\epsilon)\in(0,1)$ and $\rho_0=\rho_0(\data,
    \|H(z,|F|)\|_{L^{1+\varepsilon_0}(\Om_T)},\delta,\epsilon)\in (0,1)$ such that   if $\rho\in(0,\rho_0)$, then  
    \[\frac{1}{|Q_{\rho}^\lambda|}\iint_{Q_{V\rho}^\la}H(z,|\na u-\na \zeta|)\,dz\le \frac{1}{2^{q}3}\epsilon\la^p.\]
Also, there exists $c_\delta=c(\data,\delta)$ such that
\[
\sup_{t\in I_{8V\rho}}\fint_{B^\la_{8V\rho}}\frac{|\zeta|^2(x,t)}{(8V\rho)^2}\,dx+\fiint_{Q_{8V\rho}^\la}\left(\frac{|\zeta|^p}{(8V\la^{\frac{p-2}{2}}\rho)^p} + H(z,|\na \zeta|)\right)\,dz\le c_\delta\la^p.
\]
\end{lemma}

\begin{proof}
    We apply the standard energy estimate in \cite[Lemma~3.4]{K2024}. Testing $u-u_0-\zeta$ to 
    \[
    (u-u_0-\zeta)_t -\dv(b(\mathcal{A}(z,\na u)-\mathcal{A}(z,\na \zeta)  ))=\dv\mathcal{A}(z,F)
    \]
    in $Q_{8V\rho}^\la$, there exists $c=c(n,N,p,q,\nu,L)$ such that
    \begin{align}\label{p_c1_1}
        \begin{split}
            &\frac{1}{|I_{8V\rho}|}\sup_{t\in I_{8V\rho}}\fint_{B^\la_{8V\rho}}|u-u_0-\zeta|^2(x,t)\,dx+\fiint_{Q_{8V\rho}^\la} H(z,|\na u-\na \zeta|)\,dz\\
            &\le c\fiint_{Q_{8V\rho}^\la}H(z,|F|)\,dz.
        \end{split}
    \end{align}
    At this point, we employ \ref{a} to the right hand side of \eqref{p_c1_1}. Then it follows
    \[
    \sup_{t\in I_{8V\rho}}\fint_{B^\la_{8V\rho}}\frac{|u-u_0-\zeta|^2(x,t)}{(8V\rho)^2}\,dx+\fiint_{Q_{8V\rho}^\la} H(z,|\na u-\na \zeta|)\,dz\le c\delta\la^p.
    \]
   On the other side, by using triangle inequality, we obtain
   \begin{align*}
       \begin{split}
           &\sup_{t\in I_{8V\rho}}\fint_{B^\la_{8V\rho}}\frac{|\zeta|^2}{(8V\rho)^2}\,dx+\fiint_{Q_{8V\rho}^\la}\left(\frac{|\zeta|^p}{(8V\la^{\frac{p-2}{2}}\rho)^p} + H(z,|\na \zeta|)\right)\,dz\\
           &\le  c\sup_{t\in I_{8V\rho}}\fint_{B^\la_{8V\rho}}\frac{|u-u_0|^2}{(8V\rho)^2}\,dx+c \fiint_{Q_{8V\rho}^\la}\left(\frac{|u-u_0|^p}{(8V\la^{\frac{p-2}{2}}\rho)^p} + H(z,|\na u|)\right)\,dz\\
           &\qquad+ c\la^{p-2}\sup_{t\in I_{8V\rho}}\fint_{B^\la_{8V\rho}}\frac{|u-u_0-\zeta|^2}{(8V\rho)^2}\,dx+c\fiint_{Q_{8V\rho}^\la} H(z,|\na u-\na \zeta|)\,dz\\
           &\qquad+c\fiint_{Q_{8V\rho}^\la}\frac{|u-u_0-\zeta|^p}{(8V\la^{\frac{p-2}{2}}\rho)^p}\,dz.
       \end{split}
   \end{align*}
   Thus, applying Lemma~\ref{en_lem} and Poincar\'e inequality in the spatial direction to absorb the last term into the former term, it follows that
   \begin{align*}
       \begin{split}
           &\sup_{t\in I_{8V\rho}}\fint_{B^\la_{8V\rho}}\frac{|\zeta|^2}{(8V\rho)^2}\,dx+\fiint_{Q_{8V\rho}^\la}\left(\frac{|\zeta|^p}{(8V\la^{\frac{p-2}{2}}\rho)^p} + H(z,|\na \zeta|)\right)\,dz\\
           &\le c_\delta\la^p+c\delta\la^p.
       \end{split}
   \end{align*}
    As $\delta\in(0,1)$, the second inequality in this lemma follows.

    To derive the first inequality of this lemma, we omit the first term of the left hand side in \eqref{p_c1_1} and write the remaining term by using \ref{a} as follows.
    \[
    \frac{1}{|Q_{\rho}^\la|}\iint_{Q_{V\rho}}H(z,|\na u-\na \zeta|)\,dz\le c \delta K^{n+2}\la^p,
    \] 
    where we used facts that $V=9K$ and the choice of $K$ in \eqref{K_def}.
    The proof is completed if $c\delta K^{n+2}$ is smaller than $\tfrac{1}{2^q3}\ep$. Observe that
    \begin{align*}
        \begin{split}
            &\frac{1}{180(1+[a]_\alpha)}\delta^\frac{1}{n+2} K\\
            &=\left(\frac{\delta^\frac{1}{\alpha}}{|B_1|}\iint_{Q_{2\rho_0}(z_0)}H(z,|\na u|)\,dz+\delta^\frac{1}{\alpha}+\delta^{\frac{1-\alpha}{\alpha}}\iint_{Q_{2\rho_0}(z_0)}H(z,|F|)\,dz\right)^\frac{\alpha}{n+2}.
        \end{split}
    \end{align*}
    Therefore, if $\alpha\in(0,1)$, then we take $\delta=\delta(\data)$ small enough to handle the term $c\delta K^{n+2}$ less than $\tfrac{1}{2^q3}\epsilon$. On the other hand, if $\alpha=1$, then the last term of the above display cannot be small by taking $\delta$ small enough. Meanwhile, the H\"older inequality implies
    \[
    \iint_{Q_{2\rho_0}(z_0)}H(z,|F|)\,dz\le \left(\iint_{\Om_T}(H(z,|F|))^{1+\varepsilon_0}\,dz\right)^\frac{1}{1+\varepsilon_0}|Q_{2\rho_0}|^{\frac{\varepsilon_0}{1+\varepsilon_0}}.
    \]
    Hence, the desired estimate follows by taking $\delta$ small enough and then $\rho_0$ small enough.
\end{proof}

In order to employ the regularity property of constructed map, we will apply the scaling argument in the intrinsic cylinder as in \cite{MR2286632}. Recalling a weak solution $\zeta$ to
\[
\zeta_t-\dv(b(z)\mathcal{A}(z,\na \zeta))=0\quad\text{in}\quad Q_{8V\rho}^\la,
\]
we set
  \begin{align}\label{p_scale}
  \begin{split}
      &\zeta_{\la}(x,t)=\frac{1}{\la^\frac{p}{2}\rho}\zeta\bigl(\la^{\frac{p-2}{2}}\rho x,\rho^2t\bigr),\\
      &b_\la(x,t)=b\bigl(\la^{\frac{p-2}{2}}\rho x,\rho^2t\bigr),\\
      &a_\la(x,t)=\la^{q-p}a\bigl(\la^{\frac{p-2}{2}}\rho x,\rho^2t\bigr),\\ 
      &\mA_\la(z,\xi)=|\xi|^{p-2}\xi+a_\la(z)|\xi|^{q-2}\xi,\\
      &H_\la(z,s)=s^p+a_\la(z)s^q.
  \end{split}
    \end{align}
    for $(x,t)\in Q_{8V}$.
Note that $b_\la(z)$ still satisfies the ellipticity condition \eqref{ellipticity}. 
\begin{lemma}\label{p_sc_lem}
    The scaled map $\zeta_\la$ is a weak solution to
    \[
    \partial_t\zeta_\la-\dv(b_\la(z)\mathcal{A}_\la(z,\na \zeta_\la))=0\quad\text{in}\quad Q_{8V}.
    \]
    Moreover, the function $a_\la$ is $(\alpha,\alpha/2)$-H\"older continuous with $[a_\la]_\alpha\le [a]_\alpha$ and
    \[
    H_\la(z,|\na \zeta_\la|)=\frac{1}{\la^p}H(z,|\na \zeta|).
    \]
\end{lemma}
\begin{proof}
    From \eqref{coeff} and the scaling setting, it is easy to see $a_\la(z)$ is $(\alpha,\alpha/2)$-H\"older continuity and we also have
\[
[a_\la]_\alpha=\la^{q-p}\rho^\alpha[a]_\alpha\le [a]_\alpha,
\]
where we used Lemma~\ref{p_lem}. Also, the identity
\[
\fiint_{Q_{8V}}H_\la(z,|\na \zeta_\la|)\,dz=\frac{1}{\la^p}\fiint_{Q_{8V\rho}^\la}H(z,|\na \zeta|)\,dz
    \]
    directly follows from the scaling argument. Finally, the solvability of PDE is proved in \cite[Lemma 3.5]{K2024} as it is enough to replace $\rho$ in the reference by $\la^\frac{p-2}{2}\rho$ for the setting of this paper.
\end{proof}

Nevertheless, \eqref{eq} is the double-phase system, it is invariant under the scaling argument in the $p$-intrinsic cylinder with Assumption~\ref{p_ass}. We apply it to obtain the proper quantitative estimate of the higher integrability of $\zeta$.
\begin{lemma}\label{p_high}
  There exist $\varepsilon_\delta=\varepsilon(\data,\delta)$   and $c_\delta=c(\data,\delta)$ such that 
  \[
  \fiint_{Q_{4V\rho}^\la}(H(z,|\na \zeta|))^{1+\varepsilon_\delta}\,dz\le c_\delta \la^{p(1+\varepsilon_\delta)}.
  \]
\end{lemma}

\begin{proof}
    Recalling the center point of $Q_{8V}^\la$ and $Q_{8V}$ is $\om$, we observe from \ref{i} and Lemma~\ref{p_sc_lem} that
    \begin{align*}
        \begin{split}
            \|a_\la\|_{L^\infty(Q_{8V})}
            &\le a_\la(\om)+[a_\la]_{\alpha}(8V)^\alpha\\
            &\le \la^{q-p}a(\om)+8V[a]_{\alpha}\\
            &\le K^2+8V[a]_{\alpha}.
        \end{split}
    \end{align*}
    On the other hand, it follows from Lemma~\ref{p_c1} and Lemma~\ref{p_sc_lem} that
    \[
    \fiint_{Q_{8V}}H(z,|\na\zeta_\la|)\,dz\le c_\delta=c(\data,\delta),
    \]
    We now apply Theorem~\ref{higher} to $\zeta_\la$. Then we have
    \begin{align*}
        \begin{split}
            \fiint_{Q_{4V}}(H_\la(z,|\na \zeta_\la|))^{1+\varepsilon_\delta}\,dz
            &\le c_{\delta}\left(\fiint_{Q_{8V}}H_\la(z,|\zeta_\la|)\,dz\right)^{1+\frac{2q\varepsilon_\delta}
            {p(n+2)-2n)}}\\
            &\le c_\delta,
        \end{split}
    \end{align*}
    where $c_\delta=c(\data,\delta)$ and $\varepsilon_\delta=\varepsilon(\data,\delta)$. By scaling back, we conclude
    \[
    \fiint_{Q_{4V\rho}^\la}(H(z,|\na \zeta|))^{1+\varepsilon_\delta}\,dz\le c_\delta\la^{1+\varepsilon_\delta}.
    \]
    This completes the proof.
\end{proof}

The second map we construct is the weak solution to
\begin{align*}
    \begin{cases}
        \eta_t-\dv( b_0\mathcal{A}(z,\na \eta) )=0&\text{in}\quad Q_{4V}^\la,\\
        \eta=\zeta&\text{on}\quad\pa_pQ_{4V}^\la,
    \end{cases}
\end{align*}
where we have set
\[
b_0=(b)_{Q_{4V\rho}^\la}=(b)_{Q_{4V\rho_\om}^{\la_\om}}(\om).
\]
The following comparison estimate is a consequence of Lemma~\ref{p_high}.

\begin{lemma}\label{p_c2}
        There exists $\rho_0=\rho_0(\data,\delta,\epsilon)\in(0,1)$ such that   if $\rho\in(0,\rho_0)$, then  
    \[\frac{1}{|Q_{\rho}^\lambda|}\iint_{Q_{V\rho}^\la}H(z,|\na \zeta-\na \eta|)\,dz\le \frac{1}{2^{2q}3}\epsilon\la^p.\]
    Also, there exists $c_\delta=c(\data,\delta)$ such that 
    \[
    \sup_{t\in I_{4V\rho}}\fint_{B_{4V\rho}^\la}\frac{|\eta|^2(x,t)}{(4V\rho)^2}\,dx+\fiint_{Q_{4V\rho}^\la}\left(\frac{|\eta|^p}{(4V\la^\frac{p-2}{2}\rho)^p}+H(z,|\na \eta|)\right)\,dz\le c_\delta\la^p.\]
\end{lemma}
\begin{proof}
    By taking $\zeta-\eta$ as a test function to 
    \[
    \partial_t(\zeta-\eta)-\dv( b_0(\mathcal{A}(z,\na \zeta)-\mathcal{A}(z,\na \eta)  ) )=- \dv(  (b_0-b)\mathcal{A}(z,\na \zeta)  )
    \]
    in $Q_{4V\rho}^\la$ as in Lemma~\ref{p_c1}, we obtain
    \begin{align}\label{p_c2_1}
        \begin{split}
            &\sup_{t\in I_{4V\rho}}\fint_{B_{4V\rho}^\la}\frac{|\zeta-\eta|^2(x,t)}{(4V\rho)^2}\,dx+\fiint_{Q_{4V\rho}^\la}H(z,|\na \zeta-\na \eta|)\,dz\\
        &\le c\fiint_{Q_{4V\rho}^\la}|b_0-b(z)||\mA(z,\na \zeta)||\na \zeta-\na \eta|\,dz,
        \end{split}
    \end{align}
    where $c=c(n,N,p,q,\nu,L)$.
To estimate further, we apply Young's inequality for each $p$-Laplace part and $q$-Laplace part of $\mathcal{A}(z,\na \zeta)$. Then there holds
\begin{align*}
    \begin{split}
        &c\fiint_{Q_{4V\rho}^\la}|b_0-b(z)||\mA(z,\na \zeta)||\na \zeta-\na \eta|\,dz\\
        &\le c\fiint_{Q_{4V\rho}^\la}|b_0-b(z)||H(z,\na\zeta)|\,dz\\
        &\qquad+\frac{1}{4L}\fiint_{Q_{4V\rho}^\la}|b_0-b(z)|H(z,\na \zeta-\na \eta)\,dz.
    \end{split}
\end{align*}
Since $|b_0-b(z)|\le 2L$ holds from \eqref{ellipticity}, the last term of the above display can be absorbed into the left hand side of \eqref{p_c2_1}. Therefore it suffices to estimate the first term on the right hand side of the above display. We apply H\"older inequality and Lemma~\ref{p_high} to have
\begin{align*}
    \begin{split}
        &\fiint_{Q_{4V\rho}^\la}|b_0-b(z)||H(z,\na\zeta)|\,dz\\
        &\le \left(\fiint_{Q_{4V\rho}^\la}|b_0-b(z)|^\frac{1+\varepsilon_\delta}{\varepsilon_\delta}\,dz \right)^\frac{\varepsilon_\delta}{1+\varepsilon_\delta}\left(\fiint_{Q_{4V\rho}^\la}(H(z,\na\zeta))^{1+\varepsilon_\delta}\,dz \right)^{\frac{1}{1+\varepsilon_\delta}}\\
        &\le c_\delta \la^p\left(\fiint_{Q_{4V\rho}^\la}|b_0-b(z)|^\frac{1+\varepsilon_\delta}{\varepsilon_\delta}\,dz \right)^\frac{\varepsilon_\delta}{1+\varepsilon_\delta}.
    \end{split}
\end{align*}
    Since we have
    \[
    |b_0-b(z)|^\frac{1+\varepsilon_\delta}{\varepsilon_\delta}\le (2L)^\frac{1}{\varepsilon_\delta}|b_0-b(z)|,
    \]
   we employ \eqref{vmo} to take $\rho_0$ depending on $\data$ and $\delta$. Then \eqref{p_c2_1} becomes
   \begin{align*}
        \begin{split}
            &\sup_{t\in I_{4V\rho}}\fint_{B_{4V\rho}^\la}\frac{|\zeta-\eta|^2(x,t)}{(4V\rho)^2}\,dx+\fiint_{Q_{4V\rho}^\la}H(z,|\na \zeta-\na \eta|)\,dz\\
        &\le c_\delta \la^p\left(\fiint_{Q_{4V\rho}^\la}|b_0-b(z)|\,dz \right)^\frac{\varepsilon_\delta}{1+\varepsilon_\delta}\\
        &\le \frac{1}{(4V)^{n+2}2^{2q}3}\ep\la^{p}.
        \end{split}
   \end{align*}
   Therefore, the conclusion follows.    
\end{proof}

The regularity property we use for the next comparison estimate is a local $L^q$ estimate of $\na \eta$ by using $L^p$ norm of $\na \eta$. For this, we again adopt the scaling argument.
\begin{lemma}\label{p_lq_est}
    There exists $c_\delta=c(\data,\delta)$ such that
    \[
    \fiint_{Q_{2V\rho}^\la} |\na \eta|^q\,dz\le c_\delta \la^q.
    \]
\end{lemma}
\begin{proof}
    We consider the scaled map
    \[
    \eta_\la(x,t)=\frac{1}{\la^\frac{p}{2}\rho}  \eta\bigl(\la^{\frac{p-2}{2}}\rho x,\rho^2t\bigr) ,\quad (x,t)\in Q_{4V}.
    \]
    As $b_0$ is a constant, we employ Lemma~\ref{p_sc_lem}. Then $\eta_\la$ is a weak solution to
    \[
    \partial_t\eta_\la-\dv(b_0\mathcal{A}_\la(z,\na \eta_\la))=0\quad\text{in}\quad Q_{4V}.
    \]
    Moreover, we have from the proof of Lemma~\ref{p_high} that
    \begin{align}\label{p_eta_1}
        [a]_\la+\| a_\la\|_{L^\infty(Q_{4V})}+\fiint_{Q_{4V}}|\na \eta_\la|^{p}\,dz\le c_\delta,
    \end{align}
    while the application of the scaling argument to the estimate in Lemma~\ref{p_c2} gives
    \begin{align}\label{p_eta_2}
        \sup_{I_{4V}}\fint_{B_{4V}^\la} |\eta_\la|^2\,dx+\fiint_{Q_{4V}}|\eta_\la|^p\,dz\le c_\delta.
    \end{align}
    The conclusion of this lemma follows by scaling back from the following estimate
    \begin{align}\label{p_singer}
        \fiint_{Q_{2V}}|\na \eta_\la|^q\,dz\le c_\delta.
    \end{align}
    To show this, we divide cases.
    
    \textit{Case $\alpha\in(0,1)$:} In this case, we apply \cite[Lemma~4.2]{MR3532237} to have that for any $s\in (p,p+\tfrac{\alpha p}{n+2})$, there holds
    \[
    \iint_{Q_{2V}}|\na \eta_\la|^s\,dz \le c_\delta\left(1+ \sup_{I_{4V}}\fint_{B_{4V}^\la} |\eta_\la|^2\,dx+\fiint_{Q_{4V}}(|\eta_\la|^p+|\na \eta_\la|^p)\,dz \right)^\kappa,
    \]
     where $c_\delta=c(n,p,s,\nu,L,\alpha, V ,\delta   )$ and $\kappa=\kappa(n,p,s,\alpha)$. Since $\tfrac{\alpha p}{n+2}>\tfrac{\alpha (p(n+2)-2n)}{2(n+2)} $,
     by taking $s=q$ and using \eqref{p_eta_1} and \eqref{p_eta_2}, the estimate \eqref{p_singer} follows.

     \textit{Case $\alpha=1$:} In this case, note that $a_\la$ is ($\Tilde{\alpha},\Tilde{\alpha}/2$)-H\"older continuous for any $\Tilde{\alpha
     }\in (0,1)$. In particular, we fix $\Tilde{\alpha}$ to satisfy
     \[
     \Tilde{\alpha}>\frac{n+2}{2}-\frac{n}{p}=1- \biggl(\frac{n}{p}-\frac{n}{2}\biggr).
     \]
    Then $\tfrac{\Tilde{\alpha} p}{n+2}>\tfrac{p(n+2)-2n}{2(n+2)} $ holds and we get
   \[
    \iint_{Q_{2V}}|\na \eta_\la|^q\,dz \le c_\delta\left(1+ \sup_{I_{4V}}\fint_{B_{4V}^\la} |\eta_\la|^2\,dx+\fiint_{Q_{4V}}(|\eta_\la|^p+|\na \eta_\la|^p)\,dz \right)^\kappa,
    \]
     where $c_\delta=c(n,p,q,\nu,L,\Tilde{\alpha}, V)$ and $\kappa=\kappa(n,p,q,\Tilde{\alpha})$. Hence, \eqref{p_singer} again follows from \eqref{p_eta_1} and \eqref{p_eta_2}.
\end{proof}

The last map we construct for the comparison estimate in the $p$-intrinsic geometry is the weak solution $v\in C(I_{2V\rho};L^2(B_{2V\rho}^\la,\RR^N))\cap L^q(I_{2V\rho};W^{1,q}(B_{2V\rho}^\la,\RR^N))$ to 
\[\begin{cases}
        v_t-\dv(b_0(|\na v|^{p-2}\na v+a_s|\na v|^{q-2}\na v))=0&\text{in}\quad Q_{2V\rho}^\la,\\
        v=\eta &\text{on}\quad \pa_pQ_{2V\rho}^\la,
    \end{cases}\]
where we set
\[
a_s=\sup_{z\in Q_{2V\rho}^\la}a(z).
\]

\begin{lemma}\label{p_c3}
        There holds
       \[\frac{1}{|Q_{\rho}^\lambda|}\iint_{Q_{V\rho}^\la} H(z,|\na \eta-\na v|)\,dz\le \frac{1}{2^{2q} 3}\ep\la^p.\]
        Also, there exists $c_\delta=c(\data,\delta)$ such that
        \[\fiint_{Q_{2V\rho}^\la}\left(|\na v|^p+a_s|\na v|^q\right)\,dz\le c_\delta\la^p.\]
\end{lemma}

\begin{proof}
    We take $\eta-v$ as a test function to 
    \begin{align*}
        \begin{split}
            &\partial_t(\eta-v)-\dv( b_0(|\na \eta|^{p-2}\na \eta-|\na v|^{p-2}v +a_s (|\na \eta|^{q-2}\na \eta -|\na v|^{p-2}\na v) ) )\\
            &=- \dv(  b_0(a_s-a(z))|\na \eta|^{q-2}\na \eta
        \end{split}
    \end{align*}
    in $Q_{2V\rho}^\la$. Then we get
    \[
    \fiint_{Q_{2V\rho}^\la}\left(|\na \eta-\na v|^p+a_s|\na \eta-\na v|^q\right)\,dz\le c\fiint_{Q_{2V\rho}^\la}|a(z)-a_s||\na \eta|^{q-1}|\na \eta-\na v|\,dz
    \]
    for some $c=c(n,N,p,q,\nu,L)$. Applying \eqref{coeff} and Young's inequality, the right-hand side can be estimated by
    \begin{align*}
        \begin{split}
            &c\fiint_{Q_{2V\rho}^\la}|a(z)-a_s||\na \eta|^{q-1}|\na \eta-\na v|\,dz\\
            &\le c\fiint_{Q_{2V\rho}^\la}|a(z)-a_s||\na \eta|^{q}\,dz+\frac{1}{4}\fiint_{Q_{2V\rho}^\la}|a(z)-a_s||\na \eta-\na v|^q\,dz\\
            &\le c(V\rho)^\alpha\fiint_{Q_{2V\rho}^\la}|\na \eta|^{q}\,dz+\frac{1}{2}\fiint_{Q_{2V\rho}^\la}a_s|\na \eta-\na v|^q\,dz.
        \end{split}
    \end{align*}
   Therefore, absorbing the last term into the left hand side, it follows that
   \[
   \fiint_{Q_{2V\rho}^\la}\left(|\na \eta-\na v|^p+a_s|\na \eta-\na v|^q\right)\,dz\le c_\delta\rho^\alpha\fiint_{Q_{2V\rho}^\la}|\na \eta|^{q}\,dz.
   \]
   Moreover, we apply Lemma~\ref{p_lq_est} and Lemma~\ref{p_lem} to have
   \[
   \fiint_{Q_{2V\rho}^\la}\left(|\na \eta-\na v|^p+a_s|\na \eta-\na v|^q\right)\,dz\le \frac{1}{(2V)^{n+2}2^{2q}3}\epsilon \la^p.
   \]
    Therefore, since $a(z)\le a_s$ holds in $Q_{V\rho}^\la$, the first estimate in this lemma follows from the above inequality.
    On the other hand, we observe
    \begin{align*}
        \begin{split}
            &\fiint_{Q_{V\rho}^\la}(|\na v|^p+a_s|\na v|^q)\,dz\\
            &\le c\fiint_{Q_{V\rho}^\la}\left(|\na \eta-\na v|^p+a_s|\na \eta-\na v|^q\right)\,dz+c\fiint_{Q_{V\rho}^\la}\left(|\na \eta|^p+a_s|\na \eta|^q\right)\,dz\\
            &\le c\fiint_{Q_{V\rho}^\la}\left(|\na \eta-\na v|^p+a_s|\na \eta-\na v|^q\right)\,dz+c\fiint_{Q_{V\rho}^\la}H(z,|\na \eta|)\,dz\\
            &\qquad+c_\delta\rho^\alpha\fiint_{Q_{V\rho}^\la}|\na \eta|^q\,dz.
        \end{split}
    \end{align*}
   Hence, by using the first inequality of this lemma, Lemma~\ref{p_lq_est}, Lemma~\ref{p_lem} and Lemma~\ref{p_c2}, the second inequality of this lemma follows.
\end{proof}

\begin{lemma}\label{p_lip}
There exists $c_\delta=c(\data,\delta)$ such that
    \[\sup_{z\in Q_{V\rho}^\la}|\na v(z)|\le c_\delta\la.\]    
\end{lemma}

\begin{proof}
    We replace $a_\la(x,t)$ and $H_\la(z,s)$ in \eqref{p_scale} by the constant $\la^{q-p}a_s$ and denote
    \[
    H_\la(|\xi|)=b_0(|\xi|^p+\la^{q-p}a_s|\xi|^q)=b_0(|\xi|^{p-2}\xi+\la^{q-p}a_s|\xi|^{q-2}\xi)\cdot \xi.
    \]
    Then by Lemma~\ref{p_sc_lem}, the scaled map defined as
    \[
    v_\la(x,t)=\frac{1}{\la^\frac{p}{2}\rho}v   \bigl(\la^{\frac{p-2}{2}}\rho x,\rho^2t\bigr) ,\quad (x,t)\in Q_{2V}
    \]
    is a weak solution to
    \[
    \pa_t-\dv(b_0(|\na v_\la|^{p-2}\na v_\la+\la^{q-p}a_s|\na v_\la|^{q-2}\na v_\la))=0
    \]
    in $Q_{2V}$ with the estimate
    \[
    \fiint_{Q_{2V}}H_\la(|\na v_\la|)\,dz\le c_\delta.
    \]
    Since the application of the Lipschitz regularity in the spatial direction in \cite{BARONI2017593} gives
    \[
    \sup_{Q_{V}}|\na v_\la(z)|\le c\left(\fiint_{Q_{2V}}H_\la(|\na v_\la|)\,dz+1
    \right)^\gamma\le c_\delta
    \]
    for constants $c=c(n,p,q,\nu,L)$ and $\gamma=\gamma(n,p)$, the conclusion follows by scaling back the above inequality.
\end{proof}

Combining all the comparison estimates, we obtain the estimate below.
\begin{corollary}\label{p_com}
    There exists $\delta=\delta(\data,\epsilon)\in (0,1)$ and $\rho_0=\rho_0(\data,\|H(z,|F|)\|_{L^{1+\varepsilon_0}(\Om_T)},\delta,\epsilon)\in (0,1)$ such that if $\rho\in(0,\rho_0)$, then 
    \[
    \iint_{Q_{V\rho}^\la}H(z,|\na u-\na v|)\,dz\le \ep\la^p|Q_{\rho}^\la|.
    \]
\end{corollary}

\subsection{$(p,q)$-intrinsic case.}
We now will get comparison estimates for the case $K^2\la_\om^p< a(\om)\la_\om^q$ with the following stopping time argument in the $(p,q)$-intrinsic cylinder defined in \eqref{def_G_cylinder}. 
\begin{assumption}\label{q_ass}
   For $\om=(y,s)\in Q_{R}(z_0)$, there exist $\lao>1$ and $\rho_\om\in(0,\rho_0)$ such that $G_{16V\rho_\om}^{\la_\om}(\om)\subset Q_{2R}(z_0)$ and satisfying the following conditions.
    \begin{enumerate}[label={(\roman*)}, start=3]
        \item\label{iii} $(p,q)$-intrinsic case: $K^2\lao^p< a(\om)\lao^q$,
    \item stopping time argument for $p$-intrinsic cylinder: 
    \begin{enumerate}[label=(\alph*),series=theoremconditions, start=3]
        \item\label{c} $\fiint_{Q_{16V\rho_\om}^{\lao}(\om)}\left(H(z,|\na u|)+\delta^{-1}H(z,|F|)\right)\,dz< H(\om,\lao)$,
        \item\label{d} $\fiint_{Q_{\rho_\om}^{\lao}(\om)}\left(H(z,|\na u|)+\delta^{-1}H(z,|F|)\right)\,dz=H(\om,\lao)$,
\end{enumerate}
\end{enumerate}
\end{assumption}
For convenience, we again omit the referenced center $\om$ and $\om$ will be simply denoted by $0$.

With the assumption \ref{iii}, we prove the comparability of $a(\cdot)$ in $Q_{5V\rho}$ and thus \eqref{eq} is the $(p,q)$-Laplace type system there.
\begin{lemma}\label{coeff_com}
    We have
    \[
    \frac{a(0)}{2}\le a(z)\le 2a(0)\quad\text{for all}\quad z\in Q_{5V\rho}.
    \]
    Moreover, we have
    \[
    [a]_\alpha(5V\rho)^\alpha< \inf_{z\in Q_{5V\rho}}a(z).
    \]
\end{lemma}

\begin{proof}
    Note that the second inequality implies the first inequality. Indeed, we observe
    \[
    \sup_{z\in Q_{5V\rho}}a(z)\le \inf_{z\in Q_{5V\rho}}a(z)+[a]_\alpha(5V\rho)^\alpha\le 2\inf_{z\in Q_{5V\rho}}a(z).
    \]
    Therefore, it remains to prove the second inequality. Suppose it is false, that is, 
    \[
    \inf_{z\in Q_{5V\rho}}a(z)\le [a]_\alpha(5V\rho)^\alpha.
    \]
    Recalling \eqref{V_def}, we have
    \begin{align}\label{q_false}
     \sup_{z\in Q_{5V\rho}}a(z)\le 90K[a]_\alpha\rho^\alpha.   
    \end{align}
    On the other hand, we have from \ref{iii} and \ref{d} that
    \begin{align*}
        \begin{split}
            a(0)\la^q
            &\le \frac{\la^p+a(0)\la^q}{2\la^{\frac{n(p-2)}{2}+p}\rho^{n+2}|B_1|}\iint_{Q_\rho^\la}(H(z,|\na u|)+\delta^{-1}H(z,|F|))\,dz\\
            &\le \frac{a(0)\la^q}{\la^{\frac{n(p-2)}{2}+p}\rho^{n+2}|B_1|}\iint_{Q_\rho^\la}(H(z,|\na u|)+\delta^{-1}H(z,|F|))\,dz.
        \end{split}
    \end{align*}
     Dividing both side with $a(0)\la^q\rho^{-(n+2)}$, taking exponent $\tfrac{\alpha}{n+2}$ both side and recalling \eqref{K_def}, we obtain
     \[
     \rho^\alpha< \la^{-\frac{\alpha (p(n+2)-2n)}{2(n+2)}} \frac{1}{180[a]_\alpha}K.
     \]
    Applying \ref{iii}, \eqref{q_false} and the above inequality in order, we get
    \[
    K^2\la^p\le a(0)\la^q\le 90K[a]_\alpha \rho^\alpha \la^q\le \frac{1}{2}K^2\la^p,
    \]
    where to obtain the last inequality, we used \eqref{range_pq}. Hence this is a contradiction and the second inequality of this lemma holds.
\end{proof}

Next, we prove the corresponding result of Lemma~\ref{p_lem}.
\begin{lemma}\label{q_lem}
       For any constant $c_\delta=c(\data,\|a\|_{\infty},\|H(z,|F|)\|_{L^{1+\varepsilon_0}(\Om_T)},\delta)$, there exists $\rho_0=\rho_0(\data,\|a\|_{\infty},\|H(z,|F|)\|_{L^{1+\varepsilon_0}(\Om_T)},R,\delta,\ep)\in(0,1)$ such that if $\rho\in (0,\rho_0)$, then
\[c_\delta\rho^\alpha\la^q\le \frac{1}{(2V)^{n+2}2^{2q}3}\epsilon\la^p.\]
\end{lemma}
\begin{proof}
   The proof is also analogous to the proof of Lemma~\ref{p_lem}. Since $Q_{4R}(z_0)\subset \Om_T$, Theorem~\ref{higher} gives
   \[
   \iint_{Q_{2R}(z_0)}(H(z,|\na u|))^{1+\varepsilon_0}\,dz\le c_R,
   \]
   where $\varepsilon_0=\varepsilon_0(\data)$ and $c_R=c_R(\data,\| a\|_{L^\infty(\Om_T)}, \|H(z,|F|)\|_{L^{1+\varepsilon_0}(\Om_T) },R)$. Therefore, it follows from  \ref{d} and $G_{\rho}^\la\subset Q_{2R}(z_0)$ that
   \begin{align*}
       \begin{split}
           a(0)\la^q
           &\le\fiint_{G_{\rho}^{\la}}\left(H(z,|\na u|)+\delta^{-1}H(z,|F|)\right)\,dz\\
           &\le \left(\fiint_{G_{\rho}^{\la}}\left(H(z,|\na u|)+\delta^{-1}H(z,|F|)\right)^{1+\varepsilon_0}\,dz\right)^\frac{1}{1+\varepsilon_0}\\
           &\le c_R|G_{\rho}^{\la}|^{-\frac{1}{1+\varepsilon_0}}\\
           &\le c_R \Bigl(\la^{\frac{n(p-2)}{2}+p}(\la^p+a(0)\la^q)^{-1}\rho^{n+2}\Bigr)^{-\frac{1}{1+\varepsilon_0}}\\
           &\le c_R \Bigl(\la^{\frac{n(p-2)}{2}+p}(a(0)\la^q)^{-1}\rho^{n+2}\Bigr)^{-\frac{1}{1+\varepsilon_0}}.
       \end{split}
   \end{align*}
   Dividing both side by $a(0)\la^q\rho^{-\frac{n+2}{1+\varepsilon_0}}$ and using $\la^{p}\le a(0)\la^q$, we obtain
   \begin{align*}
       \begin{split}
           \rho^\frac{n+2}{1+\varepsilon_0}
           &\le c_R (\la^{\frac{n(p-2)}{2}+p} (a(0)\la^q)^{\varepsilon_0}  )^{-\frac{1}{1+\varepsilon_0}}\\
           &\le c_R (\la^{\frac{n(p-2)}{2}+p+\varepsilon_0p} )^{-\frac{1}{1+\varepsilon_0}}\\
           &=c_R (\la^{\frac{p(n+2)-2n}{2}+\varepsilon_0p})^{-\frac{1}{1+\varepsilon_0}}.
       \end{split}
   \end{align*}
   It follows that
   \[
   \rho^{\alpha}\le c_R\la^{-\left(\frac{\alpha(p(n+2)-2n)}{2(n+2)}+\frac{\alpha\varepsilon_0p}{n+2}\right)}
   \]
   and therefore, we apply \eqref{range_pq} to have
   \[
   c_\delta\rho^\alpha\la^q\le c_\delta c_R \rho^{\alpha\left( 1-\left(\frac{\alpha(p(n+2)-2n)}{2(n+2)}+\frac{\alpha\varepsilon_0p}{n+2}\right)^{-1}\left(\frac{\alpha(p(n+2)-2n)}{2(n+2)}\right)   \right)}\la^p.
   \]
   Observing
   \[
   1-\left(\frac{\alpha(p(n+2)-2n)}{2(n+2)}+\frac{\alpha\varepsilon_0p}{n+2}\right)^{-1}\left(\frac{\alpha(p(n+2)-2n)}{2(n+2)}\right)>0,
   \]
   we take $\rho_0$ small enough depending on the above exponent, $c_R$ and $c_\delta$ to deduce the conclusion.
\end{proof}

Let $\zeta\in C(J^\la_{4V\rho};L^2(B^\la_{4V\rho},\RR^N))\cap L^q(J_{4V\rho}^\la;W^{1,q}(B_{4V\rho}^\la,\RR^N))$ be the weak solution to
\[\begin{cases}
        \zeta_t-\dv (b(z)\mA(z,\na \zeta))=0&\text{in}\quad G_{4V\rho}^{\la},\\
        \zeta=u&\text{on}\quad\pa_p G_{4V\rho}^{\la}.
    \end{cases}\]

\begin{lemma}\label{q_c1}
    There exist $\delta=\delta(\data,\epsilon)\in(0,1)$ and $\rho_0=\rho_0(\data,
    \|H(z,|F|)\|_{L^{1+\varepsilon_0}(\Om_T)},\delta,\epsilon)\in (0,1)$ such that
    \[\frac{1}{|G_{\rho}^\lambda|}\iint_{G_{V\rho}^\la} H(z,|\na u-\na \zeta|)\,dz\le \frac{1}{2^{q} 3}\epsilon H(0,\la).\]
    Also, there exits $c=c(n,N,p,q,\nu,L)$ such that
    \[\fiint_{G_{4V\rho}^\la} H(z,|\na\zeta|)\,dz\le cH(0,\la).\]
\end{lemma}

\begin{proof}
As in Lemma~\ref{p_c1}, we test $u-\zeta$ to 
\[
(u-\zeta)_t-\dv(b(\mathcal{A}(z,\na u)-\mathcal{A}(z,\na \zeta)))=\dv\mathcal{A}(z,F)
\]
in $G_{4V\rho}^\la$ and obtain
\[
\fiint_{G_{4V\rho}^\la} H(z,|\na u-\na \zeta|)\,dz\le c\fiint_{G_{4V\rho}^\la}H(z,|F|)\,dz\le c\delta H(0,\la),
\]
where $c=c(n,N,p,q,\nu,L)$. Following the same argument in the proof of Lemma~\ref{p_c1}, the triangle inequality and \ref{c} give
\begin{align*}
    \begin{split}
        \fiint_{G_{4V\rho}^\la} H(z,|\na\zeta|)\,dz
        &\le c\fiint_{G_{4V\rho}^\la} H(z,|\na\zeta-\na u|)\,dz+c\fiint_{G_{4V\rho}^\la} H(z,|\na u|)\,dz\\
        &\le c\fiint_{G_{4V\rho}^\la}(H(z,|F|)+H(z,|\na u|))\,dz\\
        &\le cH(0,\la).
    \end{split}
\end{align*}
On the other hand, the estimate for the right hand side of
\[
\frac{1}{|G_{\rho}^\la|}\iint_{G_{4V\rho}^\la} H(z,|\na u-\na \zeta|)\,dz\le cV^{n+2}\delta H(0,\la)
\]
is the same as in the proof of Lemma~\ref{p_c1}. We omit the details.
\end{proof}

Next, consider the weak solution $\eta\in C(J^\la_{4V\rho};L^2(B_{4V\rho},\RR^N))\cap L^q(J^\la_{4V\rho};W^{1,q}(B_{4V\rho},\RR^N))$ to 
\begin{align*}
    \begin{cases}
        \eta_t-\dv(b(z)\mA(0,\na \eta))=0&\text{in}\quad G_{4V\rho}^\la,\\
        \eta=\zeta&\text{on}\quad \pa_p G_{4V\rho}^\la.
    \end{cases}
\end{align*}

\begin{lemma}\label{q_c2}
       There exists $\rho_0=\rho_0(\data,\|a\|_{\infty},\|H(z,|F|)\|_{L^{1+\varepsilon_0}(\Om_T)},\delta,\epsilon)\in(0,1)$ such that if $\rho\in (0,\rho_0)$, then
    \[\frac{1}{|G_{\rho}^\lambda|}\iint_{G_{V\rho}^\la}H(z,|\na \zeta-\na \eta|)\,dz\le \frac{1}{2^{2q}3}\epsilon H(0,\la).\]
    Also, there exists $c=c(n,N,p,q,\nu,L)$ such that
    \[
    \fiint_{G_{4V\rho}^\la}|\na \eta|^q\,dz\le c \la^q.\]
\end{lemma}
\begin{proof}
    Again by taking $\zeta-\eta$ as a test function to 
    \[
    (\zeta-\eta)_t-\dv(b(\mathcal{A}(0,\na \zeta) -\mathcal{A}(0,\na \eta) ))=\dv(b (a(0)-a(z))|\na \zeta|^{q-2}\na \zeta)
    \]
    in $G_{4V\rho}^\la$ and following the proof in Lemma~\ref{p_c3}, we get
    \[
    \fiint_{G_{4V\rho}^\la}H(0,|\na \zeta-\na \eta|)\,dz\le c\fiint_{G_{4V\rho}^\la}b(z)|a(0)-a(z)||\na \zeta|^q\,dz.
    \]
    Note that by \ref{iii}, \ref{c}, Lemma~\ref{coeff_com} and Lemma~\ref{q_c1}, we have
    \begin{align*}
        \begin{split}
            \fiint_{G_{4V\rho}^\la}a(0)|\na \zeta|^q\,dz
            &\le \fiint_{G_{4V\rho}^\la}2a(z)|\na \zeta|^q\,dz\\
            &\le cH(0,\la)\\
            &\le c a(0)\la^q.
        \end{split}
    \end{align*}
   Therefore we obtain
   \[
   \fiint_{G_{4V\rho}^\la}|\na \zeta|^q\,dz\le c\la^q.
   \]
    Applying \eqref{ellipticity}, \eqref{coeff} and the above inequality, it follows that
    \[
    \fiint_{G_{4V\rho}^\la}H(0,|\na \zeta-\na \eta|)\,dz\le c(V\rho)^\alpha\la^q.
    \]
    Moreover, the first inequality of this lemma follows from Lemma~\ref{coeff_com} and Lemma~\ref{q_lem}. Meanwhile, the second inequality also follows from the triangle inequality and the above estimates.
\end{proof}

To derive the comparison estimate with the frozen coefficient $b(z)$, we will again employ the estimate of the higher integrability. To do this, we set
  \begin{align*}
  \begin{split}
      &\eta_{\la}(x,t)=\tfrac{1}{\la^\frac{p}{2}\rho}\eta\bigl(\la^{\frac{p-2}{2}}\rho x,\tfrac{\la^p}{H(0,\la)}\rho^2t\bigr),\\
      &b_\la(x,t)=b\bigl(\la^{\frac{p-2}{2}}\rho x,\tfrac{\la^p}{H(0,\la)}\rho^2t\bigr),\\
      &\mA_\la(0,\xi)=\tfrac{\la}{H(0,\la)}(\la^{p-1}|\xi|^{p-2}\xi+a(0)\la^{q-1}|\xi|^{q-2}\xi),\\
  \end{split}
    \end{align*}
    for $(x,t)\in Q_{4V}$.
\begin{lemma}\label{q_sc_lem}
    The scaled map $\eta_\la$ is  a weak solution to
    \[
    \pa_t \eta_\la-\dv(b_\la(z)\mathcal{A}_\la(0,\na \eta_\la))=0\quad\text{in}\quad Q_{4V}.
    \]
    Moreover, we have
    \[
    \fiint_{Q_{4V}}|\na \eta_\la|^q\,dz=\frac{1}{\la^q}\fiint_{G_{4V}^\la}|\na \eta|^q\,dz.
    \]
\end{lemma}

\begin{proof}
    The proof is in \cite[Lemma~3.16]{K2024}. It is enough to replace $\rho$ therein by $\la^{\frac{p-2}{2}}\rho$ for this intrinsic geometry.
\end{proof}

\begin{lemma}\label{q_high}
    There exist $\varepsilon_0=\varepsilon_0(n,N,q,\nu,L)$   and $c=c(n,N,p,q,\nu,L)$ such that 
  \[
  \fiint_{G_{2V\rho}^\la}|\na \eta|^{q(1+\varepsilon_0)}\,dz\le c \la^{q(1+\varepsilon_0)}.
  \]
\end{lemma}

\begin{proof}
    Note that by applying \ref{iii}, we have
    \[
    \frac{1}{2}|\xi|^q\le\frac{a(0)\la^q}{H(0,\la)}|\xi|^q\le\frac{\la^p}{H(0,\la)}|\xi|^p+\frac{a(0)\la^q}{H(0,\la)}|\xi|^q=\mathcal{A}_\la(0,\xi)\cdot\xi 
    \]
    and similarly, we also have
    \[
    \mathcal{A}_\la(0,\xi)\cdot\xi\le  \frac{\la^p}{\la^p}|\xi|^p+\frac{a(0)\la^q}{a(0)\la^q}|\xi|^q\le 2(|\xi|+1)^q.
    \]
    Therefore $\mathcal{A}_\la(0,\xi)$ is $q$-Laplace type operator. The higher integrability of parabolic $p$-Laplace system in \cite{MR1749438} leads to
    \[
    \fiint_{Q_{2V}}|\na \eta_\la|^{q(1+\varepsilon_0)}\,dz\le c\left( \fiint_{Q_{4V}}|\na \eta_\la|^{q}\,dz+1 \right)^{1+\frac{2q\varepsilon_0}{q(n+2)-2n}},
    \]
    where $c=c(n,N,q,\nu,L)$ and $\varepsilon_0=\varepsilon_0(n,N,q,\nu,L)$. Since the right hand side of the above inequality is bound above by $c=c(n,N,p,q,\nu,L)$ with the application of Lemma~\ref{q_c2} and Lemma~\ref{q_sc_lem}, the conclusion follows by scaling back on the left hand side.
\end{proof}

Finally, let $v\in C(J^\la_{2V\rho};L^2(B_{2V\rho}^\la,\RR^N))\cap L^q(J^\la_{2V\rho};W^{1,q}(B_{2V\rho}^\la,\RR^N))$ be the weak solution  to 
\[\begin{cases}
        v_t-\dv(b_0(\mathcal{A}(0,\na v)))=0&\text{in}\quad G_{2V\rho}^\la,\\
        v=\eta &\text{on}\quad \pa_pG_{2V\rho}^\la,
    \end{cases}\]
where
\[
b_0=(b)_{G_{2V\rho}^\la}.
\]

\begin{lemma}\label{q_c3}
    There exists $\rho_0=\rho_0(n,N,p,q,\nu,L,\epsilon)$ such that if $\rho\in(0,\rho_0)$, then
    \[
    \frac{1}{|G_\rho^\la|}\iint_{G_{V\rho}^\la}H(z,|\na \eta-\na v|)\, dz\le \frac{1}{2^{2q}3}\epsilon H(0,\la).
    \]
    Moreover, we have
    \[
    \fiint_{G_{2V\rho}^\la}|\na   v  |^{q}\,dz\le c\la^q.
    \]
\end{lemma}

\begin{proof}
    The proof is analogous to the proof of Lemma~\ref{p_c2} by replacing $\zeta$, $\eta$ and $\mathcal{A}(z,\xi)$ by $\eta$, $v$ respectively and $\mathcal{A}(0,\xi)$ and applying Lemma~\ref{q_high} instead for the higher integrability. We omit the details.
\end{proof}

Again, the Lipschitz regularity of $v$ is as follows.
\begin{lemma}
    There exists $c=c(n,N,p,q,\nu,L)$ such that
    \[\sup_{z\in G_{V\rho}^\la}|\na v(z)|\le c\la.\]    
\end{lemma}

\begin{proof}
    Denoting the scaled map
    \[
    v_\la=\frac{1}{\la^\frac{p}{2}\rho}  v    \left(\la^\frac{p-2}{2}\rho x,\frac{\la^p}{H(0,\la)}\rho^2t  \right)\quad\text{for}\quad (x,t)\in Q_{2V},
    \]
    we deduce from Lemma~\ref{q_sc_lem} and Lemma~\ref{q_c3} that $v_\la$ is a weak solution to
    \[
    \partial_t v_\la -\dv(b_0\mathcal{A}_\la(0,\na v_\la)) =0\quad\text{in}\quad Q_{2V}
    \]
    with the estimate
    \[
    \fiint_{Q_{2V}}|\na v_\la|^q\,dz\le c
    \]
    for $c=c(n,N,p,q,\nu,L)$. Therefore, for the functional defined as
    \[
    H_\la(0,|\xi|)=b_0\left(   \frac{\la^p}{H(0,\la)}|\xi|^{p}+\frac{a(0)\la^q}{H(0,\la)}|\xi|^q   \right) =b_0\mathcal{A}_\la(0,\xi)\cdot \xi,
    \]
    it follows that
    \[
    \fiint_{Q_{2V}}H_\la(0,|\na v_\la|)\, dz\le c\fiint_{Q_{2V}} |\na v_\la|^p+|\na v_\la|^q\,dz\le c
    \]
    for $c=c(n,N,p,q,\nu,L)$. Hence the conclusion follows as in Lemma~\ref{p_lip}.
\end{proof}

As in the $p$-intrinsic case, we end this subsection with the following estimate.
\begin{corollary}
    There exist $\delta=\delta(\data,\epsilon)\in (0,1)$ and $\rho_0=\rho_0(\data,\|H(z,|F|)\|_{L^{1+\varepsilon_0}(\Om_T)},\delta,\epsilon)\in (0,1)$ such that if $\rho\in(0,\rho_0)$, then 
    \[
    \iint_{G_{V\rho}^\la}H(z,|\na u-\na v|)\,dz\le \ep\la^p|G_{\rho}^\la|.
    \]
\end{corollary}

\section{Stopping time arguments}
In this section, we will verify Assumption~\ref{p_ass} and Assumption~\ref{q_ass} by using the stopping time argument and prove the Vitali covering argument for intrinsic cylinders with covering constant $V=9K$, see \eqref{K_def} and \eqref{V_def}.

To begin with, we recall the referenced cylinder $Q_{2\rho}(z_0)\subset \Om_T$ where $\rho\in(0,\rho_0)$ and $\rho_0$ will be determined as $\epsilon$ is chosen.
We denote
\[
\la_0^\frac{p(n+2)-2n}{2}=\fiint_{Q_{2  \rho  }(z_0)}(H(z,|\na u|)+\delta^{-1}H(z,|F|))\,dz+1
\]
and
\[
\La_0=\la_0^p+\|a\|_{L^\infty(\Om_T)}\la_0^q.
\]
For any $r\in(0,2\rho)$, we denote upper level sets
\begin{align*}
    \begin{split}
        &\Psi(\La,r)=\{ z\in Q_{r}(z_0): H(z,|\na u(z)|)>\La  \},\\
        &\Phi(\La,r)=\{ z\in Q_{r}(z_0): H(z,|F(z)|)>\La  \}.
    \end{split}
\end{align*}
In order to utilize the technical lemma in the next section, we take $r_1,r_2$ such that
\[
\rho\le r_1<r_2\le2\rho
\]
and consider the level
\begin{align}\label{La}
    \La> \left( \frac{32V\rho}{r_2-r_1} \right)^\frac{2q(n+2)}{p(n+2)-2n}\La_0,
\end{align}
where the term with the exponent on the right hand side is bigger than 1. In this  section, we fix $\La$ satisfying \eqref{La}.

Now, for each Lebesgue point $\om\in \Psi(\La,r_1)$, let $\la_\om$ be defined as
\begin{align}\label{la}
    \La=\la_\om^p+a(\om)\la_\om^q.
\end{align}
Since the function $0<s\mapsto s^p+a(\om)s^q$ is strictly increasing continuous function with
\[
\lim_{s\to0^+}s^p+a(\om)s^q=0,\quad \lim_{s\to\infty}s^p+a(\om)s^q=\infty,
\]
$\la_\om$ uniquely exists. Furthermore, there holds
\begin{align}\label{lao}
    \lao>\left( \frac{32V\rho}{r_2-r_1} \right)^\frac{2(n+2)}{p(n+2)-2n}\la_0.
\end{align}
Indeed, if the above inequality fails, then we get the following contradiction
\[
\La=\lao^p+a(\om)\lao^q\le \left( \frac{32V\rho}{r_2-r_1} \right)^\frac{2q(n+2)}{p(n+2)-2n}(\la_0^p+a(\om)\la_0^q)\le \La_0.
\]

Along with above settings, we are ready to apply the stopping time argument.
\begin{lemma}\label{p_stopping}
    Let $\om\in \Psi(\La,r_1)$ be a Lebesgue point and $\lao$ be defined in \eqref{lao}. Then there exists  stopping time $\rho_\om$ such that
    \[
    0<\rho_\om <\frac{r_2-r_1}{16V}
    \]
    satisfying
    \begin{align*}
        \begin{split}
            &\fiint_{Q_{r}^{\lao}(\om)} (H(z,|\na u|)+\delta^{-1}H(z,|F|))\,dz<\lao^p,\\
            &\fiint_{Q_{\rho_\om}^{\lao}(\om)} (H(z,|\na u|)+\delta^{-1}H(z,|F|))\,dz=\lao^p
        \end{split}
    \end{align*}
    for $r\in(\rho_\om, r_2-r_1 )$. Moreover, there holds
    \[
    \lao\le \left( \frac{2\rho}{\rho_\om}\right)^\frac{p(n+2)-2n}{2(n+2)}\la_0.
    \]
\end{lemma}
\begin{proof}
    Since $\om\in Q_{r_1}(z_0)\subset Q_{2\rho}(z_0)\subset \Om_T$, note that
    $Q_{r_2-r_1}(\om)\subset Q_{2\rho}(z_0)$.
    For any $r$ such that
    \[
    \frac{r_2-r_1}{16V}<r<r_2-r_1,
    \]
    we observe 
    \begin{align*}
        \begin{split}
            &\fiint_{Q_{r}^{\lao}(\om)} (H(z,|\na u|)+\delta^{-1}H(z,|F|))\,dz\\
            &\le \frac{|Q_{2\rho}|}{|Q_{r}^\la|}\fiint_{Q_{2\rho}(z_0)} (H(z,|\na u|)+\delta^{-1}H(z,|F|))\,dz\\
            &\le\frac{(2\rho)^{n+2}}{\lao^\frac{n(p-2)}{2} r^{n+2}}\la_0^\frac{p(n+2)-2n}{2}\\
            &\le \left( \frac{32V\rho}{r_2-r_1} \right)^{n+2}\lao^\frac{n(2-p)}{2}\la_0^\frac{p(n+2)-2n}{2}.
        \end{split}
    \end{align*}
    Recalling $p>\tfrac{2n}{n+2}$ and \eqref{lao} holds, we get
    \[
    \fiint_{Q_{r}^{\lao}(\om)} (H(z,|\na u|)+\delta^{-1}H(z,|F|))\,dz< \lao^p.
    \]
    On the other hand, since $\om\in \Psi(\La,r_1)$, it follows from \eqref{la} that $|\na u(\om)|>\lao$. As we have $\lao^p<|\na u(\om)|^p\le H(\om,|\na u(\om)|)$, there holds
    \[
    \lim_{r\to0^+} \fiint_{Q_{r}^{\lao}(\om)} (H(z,|\na u|)+\delta^{-1}H(z,|F|))\,dz>\lao^p.
    \]
    As the integral is continuous with respect to $r$, there exists a stopping time $\rho_\om\in (0, (16V)^{-1}(r_2-r_1))$ fulfilling conditions in the statement of this lemma. To prove the last inequality of the lemma, we observe
    \begin{align*}
        \begin{split}
            \lao^p
            &=\fiint_{Q_{\rho_\om}^{\lao}(\om)} (H(z,|\na u|) +H(z,|F|) )\,dz\\
            &\le \frac{|Q_{2\rho}|}{|Q_{\rho_\om}^{\lao}|}\fiint_{Q_{2\rho}(z_0)} (H(z,|\na u|) +H(z,|F|) )\,dz\\
            &\le \left(\frac{2\rho}{\rho_\om}\right)^{n+2}\lao^\frac{n(2-p)}{2}\la_0^\frac{p(n+2)-2n}{2}.
        \end{split}
    \end{align*}
   Therefore, we obtain
   \[
   \rho_\om^{n+2}\le \left(\frac{\la_0}{\lao}\right)^\frac{p(n+2)-2n}{2} (2\rho)^{n+2}
   \]
\end{proof}

If $p$-intrinsic case $K^2\lao^p\ge a(\om)\lao^q$ holds, then Lemma~\ref{p_stopping} guarantees Assumption~\ref{p_ass}. Meantime, if $(p,q)$-intrinsic case $K^2\lao^p< a(\om)\lao^q$ holds, then we again apply the stopping time argument with the $(p,q)$-intrinsic cylinder.
\begin{lemma}\label{q_stopping}
    Let $\om\in \Psi(\La,r_1)$ be a Lebesgue point and $\lao$ be defined in \eqref{lao}. Suppose $(p,q)$-intrinsic case $K^2\lao^p< a(\om)\lao^q$ holds. Then there exists  stopping time $\varrho_\om$ such that
    \[
    0<\varrho_\om <\rho_\om
    \]
    satisfying
    \begin{align*}
        \begin{split}
            &\fiint_{G_{r}^{\lao}(\om)} (H(z,|\na u|)+\delta^{-1}H(z,|F|))\,dz<\La,\\
            &\fiint_{G_{\varrho_\om}^{\lao}(\om)} (H(z,|\na u|)+\delta^{-1}H(z,|F|))\,dz=\La
        \end{split}
    \end{align*}
    for $r\in(\varrho_\om, r_2-r_1 )$. Moreover, there holds
    \[
    \lao\le \left( \frac{2\rho}{\varrho_\om}\right)^\frac{p(n+2)-2n}{2(n+2)}\la_0.
    \]
\end{lemma}

\begin{proof}
Since $a(\om)>0$, we have $\lao^p<H(\om,\lao)=\La$. Therefore, it follows that for any $r>0$, we have
\[
G_r^{\lao} \subset Q_r^{\lao},\quad G_r^{\lao} \ne Q_r^{\lao}.
\]
For any $r\in [\rho_\om,r_2-r_1)$, we have from Lemma~\ref{p_stopping} that
    \begin{align*}
        \begin{split}
            &\fiint_{G_{r}^{\lao}(\om)} (H(z,|\na u|)+\delta^{-1}H(z,|F|))\,dz\\
            &< \frac{|Q_r^{\lao}|}{|G_r^{\lao}|}\fiint_{Q_{r}^{\lao}(\om)} (H(z,|\na u|)+\delta^{-1}H(z,|F|))\,dz\\
            &\le \frac{H(\om,\lao)}{\lao^p}\lao^p\\
            &=H(\om,\lao).
        \end{split}
    \end{align*}
As $\La<H(\om,|\na u(\om)|)$ holds, we get
\[
\lim_{r\to0^+}\fiint_{G_{r}^{\lao}(\om)} (H(z,|\na u|)+\delta^{-1}H(z,|F|))\,dz>\La.
\]
Again by the continuity of integral in the radius $r$, there exists a stopping time $\varrho_\om$ such that the conclusion of the lemma holds. Furthermore, the last inequality of this lemma follows from Lemma~\ref{p_stopping} as $\varrho_\om<\rho_\om$.
\end{proof}

The previous lemma proves the conditions in Assumption~\ref{q_ass} by replacing $\rho_\om$ there in by $\varrho_\om$.

In the rest of this paper, we will use the following notation. For $z\in \Psi(\La,r_1)$, we write
\[
Q_z=
\begin{cases}
    Q_{l_z}^{\la_z}(z)&\text{if $p$-intrinsic case,}\\
    G_{l_z}^{\la_z}(z)&\text{if $(p,q)$-intrinsic case,}
\end{cases}
\]
where
\[
l_z=
\begin{cases}
    \rho_z&\text{if $p$-intrinsic case,}\\
    \varrho_z&\text{if $(p,q)$-intrinsic case.}
\end{cases}
\]
Since the scaling factors are pointwise, the comparability of $\la_{(\cdot)}$ is necessary to prove the Vitali covering lemma.
\begin{lemma}\label{la_com}
    Let $\om,z\in \Psi(\La,r_1)$ be Lebesgue points with $ Q_\om\cap  Q_z\ne\emptyset$. Then for $\lao$ and $\la_z$ defined in \eqref{la}, we have
    \[
    2^{-\frac{1}{p}}\la_z\le \lao\le 2^\frac{1}{p}\la_z.
    \]
\end{lemma}
\begin{proof}
    It is suffice to show $\lao\le 2^\frac{1}{p}\la_z$. For the proof, we divide cases.

    \textit{Case $K^2\lao^p\ge a(\om)\lao^q$.} We prove by contradiction. Suppose 
    \begin{align}\label{la_com_hy}
        \lao> 2^\frac{1}{p}\la_z.
    \end{align}
    Using the above inequality and \eqref{coeff}
    \[
    \La=\la_z^p+a(z)\la_z^q< \frac{1}{2}\lao^p+\frac{1}{2}a(z)\lao^q \le \frac{1}{2}(\lao^p+a(\om)\lao^q)+[a]_\alpha\rho_\om^\alpha\lao^q.
    \]
    On the other hand, we have from Lemma~\ref{p_lem} that $[a]_\alpha\rho_\om^\alpha\lao^q \le \tfrac{1}{2}\lao^p$ and therefore we conclude
    \[
    \La< \frac{1}{2}\La+\frac{1}{2}\lao^p\le \La.
    \]
    This is a contradiction and \eqref{la_com_hy} is false.

    \textit{Case $K^2\lao^p< a(\om)\lao^q$.} The proof for this case is analogous. The same argument holds with replacing $\rho_\om$ by $\varrho_\om$ and Lemma~\ref{p_lem} by Lemma~\ref{q_lem}.

    This completes the proof.
\end{proof}

We now state the Vitali covering lemma.
\begin{lemma}\label{Vitali}
    There exists a pairwise disjoint set $\{Q_i\}_{i\in\mathbb{N}}$ where $Q_{i}=Q_{z_i}$ for Lebesgue points $z_i\in \Psi(\La,r_1)$ such that for any Lebesgue point $z\in \Psi(\La,r_1)$ with $Q_z$, we have
    \[
    Q_{z}\subset VQ_i
    \]
    for some $i\in\mathbb{N}$ where we denoted the scaled cylinder by
    \[
    d Q_z=
\begin{cases}
    Q_{d l_z}^{\la_z}(z)&\text{if $p$-intrinsic case,}\\
    G_{d l_z}^{\la_z}(z)&\text{if $(p,q)$-intrinsic case,}
\end{cases}
    \]
    for any $d>0$.
\end{lemma}
\begin{proof}
    We denote the family of intrinsic cylinders having the Lebesgue point as the center by
    \[
    \mathcal{F}=\left\{  Q_z:z\in\Psi(\La,r_1) \right\}
    \]
    and for each $j\in\mathbb{N}$, consider its subfamily
    \[
    \mathcal{F}_j=\left\{  Q_z\in \mathcal{F}:  \frac{r_2-r_1}{16V 2^j}<l_z\le \frac{r_2-r_1}{16V2^{j-1}}  \right\}.
    \]
    Note that if for all $Q_z\in \mathcal{F}_j$, the quantity $\la_z$ is bounded below by $\la_0$ as well as bounded above uniformly since the radius is bounded below and Lemma~\ref{p_stopping} and Lemma~\ref{q_stopping} hold.

    We take $\mathcal{D}_1$ as a maximal disjoint collection of cylinders in $\mathcal{F}_1$. As the scaling factors $\la_{(\cdot)}$ and radius are uniformly bounded below and above by positive numbers, $\mathcal{D}_1$ is finite. Inductively, for chosen $\mathcal{D}_1,...,\mathcal{D}_j$, we select a maximal disjoint subset
    \[
    \mathcal{D}_{j+1}=\left\{  Q_z\in\mathcal{F}_{j+1}: Q_\om\cap Q_z\ne \emptyset\quad\text{for all}\quad Q_\om \in \cup_{1\le k\le j}  \mathcal{D}_k\right\}.
    \]
    Then since each $\mathcal{D}_j$ contains finite cylinders, we rearrange the subfamily
    \[
    \mathcal{D}=\bigcup_{j\in\mathbb{N}} \mathcal{D}_j,
    \]
    and denote it by $\{Q_i\}_{i\in \mathbb{N}}$. 
    
    In the remaining of the proof, we will show the following claim. For any $Q_z\in\mathcal{F}$, there exists $Q_\om\in \mathcal{D}$ such that
    \[
    Q_z\cap Q_\om\ne\emptyset\quad\text{and}\quad Q_z\subset VQ_\om.
    \]
    To start with, we note that $Q_z\in \mathcal{F}$ implies $Q_z\in\mathcal{F}_j$ for some $j\in\mathbb{N}$. Therefore, by the maximal disjoint property of $\mathcal{D}_j$, there exists $Q_\om\in \mathcal{D}_j$ such that
    \[
    Q_{z}\cap Q_\om\ne\emptyset.
    \]
    Moreover, by the construction of $\mathcal{F}_j$, there holds
    \begin{align}\label{rad_com}
        l_z\le 2l_\om.
    \end{align}
    As a result, we have
    \begin{align}\label{st_vitali}
        Q_{l_z}(z)=B_{l_z}(x)\times I_{l_z}(t)\subset Q_{5l_\om}(\om)=B_{5l_\om}(y)\times I_{5l_\om}(s),
    \end{align}
    where $(x,t)$ and $(y,s)$ are projections of $z$ and $\om$ respectively on the spatial direction and the time direction.
    To prove the inclusion part of the claim, we divide cases.

   \textit{Case $Q_z$ and $Q_\om$ are $p$-intrinsic}. 
   We observe
   \[
   Q_z=B_{l_z}^{\la_z}(x)\times I_{l_z}(t),\qquad Q_\om=B_{l_\om}^{\lao}(y)\times I_{l_z}(s).
   \]
   Thus the time inclusion directly follows from \eqref{st_vitali} as we have set $5\le V=9K$. On the other hand, to see the inclusion in the spatial direction, we apply Lemma~\ref{la_com} and \eqref{rad_com} to have
   \[
   \la_z^\frac{p-2}{2}l_z\le 2^{\frac{2-p}{2p}+1}\la_\om^\frac{p-2}{2} l_\om\le 2^2\la_\om^\frac{p-2}{2} l_\om.
   \]
   It follows that
   \[
   B_{l_z}^{\la_z}(x)\subset B_{9l_\om}^{\lao}(y)\subset B_{Vl_\om}^{\lao}(y)
   \]
   and therefore, the claim holds for this case.

   \textit{Case $Q_z$ is $p$-intrinsic and $Q_\om$ is $(p,q)$-intrinsic}.
   We have
   \[
   Q_z=B_{l_z}^{\la_z}(x)\times I_{l_z}(t),\qquad Q_\om=B_{l_\om}^{\lao}(y)\times J_{l_\om}^{\lao}(s)
   \]
   For the spatial direction, we follow the argument in the first case and obtain
   \[
   B_{l_z}^{\la_z}(x)\subset B_{Vl_\om}^{\lao}(y).
   \]
   Meanwhile, to obtain the time inclusion part, we employ $a(z)\la_z^q \le K^2\la_z^p$ and Lemma~\ref{la_com} to have
    \[
    l_z^2= \frac{\La}{\La}l_z^2\le \frac{2K^2\la_z^p}{\La}l_z^2\le 16K^2\frac{\lao^p}{\La}l_\om^2.
    \]
    Therefore, we obtain
    \[
    I_{l_z}(t)\subset J^{\lao}_{6Kl_\om}(s)\subset J^{\lao}_{Vl_\om}(s)
    \]
    and the claim is proved.

    \textit{Case $Q_z$ is $(p,q)$-intrinsic and $Q_\om$ is $p$-intrinsic}.
    Since we have
    \[
    Q_z=B^{\la_z}_{l_z}(x)\times J^{\la_z}_{l_z}(t),\qquad Q_{\om}=B^{\lao}_{l_\om}(y)\times I_{l_\om}(s),
    \]
    the inclusion in the spatial direction holds as the first case while since $J^{\la_z}_{l_z}(t)\subset I_{l_z}(t)$, the inclusion in time direction holds by \eqref{st_vitali}. This completes the proof for this case.

    \textit{Case $Q_z$ and $Q_\om$ are $(p,q)$-intrinsic}.
    In order to prove the inclusion for
    \[
    Q_z=B_{l_z}^{\la_z}(x)\times J^{\la_z}_{l_z}(t),\qquad Q_\om=B_{l_\om}^{\lao}(y)\times J^{\lao}_{l_\om}(s),
    \]
     we again enough to check the inclusion in the time direction as the inclusion in the spatial direction is the same as in the first case. Since Lemma~\ref{la_com} and \eqref{rad_com} give
     \[
     \frac{\la_z^p}{\La}l_z^2\le 8\frac{\lao^p}{\La}l_\om^2,
     \]
     we obtain
     \[
     J^{\la_z}_{l_z}(t)\subset J^{\lao}_{Vl_\om}(s).
     \]
     Hence, the proof is completed.
    
\end{proof}

\section{Proof of Theorem~\ref{main_theorem}}
In this section, we prove the main theorem. The following lemma will be used in the end of the proof. For the proof, see \cite[Lemma 8.3]{MR1962933}.
\begin{lemma}\label{iter_lem}
	Let $0<r<R<\infty$ and $h:[r,R]\longrightarrow\RR$ be a non-negative and bounded function. Suppose there exist $\vartheta\in(0,1)$, $A,B\ge0$ and $\gamma>0$ such that
	\[h(r_1)\le \vartheta h(r_2)+\frac{A}{(r_2-r_1)^\gamma}+B
		\quad\text{for all}\quad
		0<r\le r_1<r_2\le R.\]
	Then there exists a constant $c=c(\vartheta,\gamma)$ such that
	\[h(r)\le c\left(\frac{A}{(R-r)^\gamma}+B\right).\]
\end{lemma}

We recall that if $\epsilon$ is chosen, then $\delta$ and $K$ will be determined and finally $\rho_0$ will be selected as in Section~\ref{com_sec}.
\begin{proof}[Proof of Theorem~\ref{main_theorem}]
To begin with, we denote
\[
\kappa=\frac{1}{4(K^2+1)}.
\]
For each $\La$ satisfying \eqref{La}, we consider the pairwise disjoint set $\{Q_i\}_{i\in\mathbb{N}}$ from Lemma~\ref{Vitali} and denote each scaling factor of cylinder $Q_i$ as
\[
\la_i=\la_{z_i}.
\]
For each $i$, we will employ estimates in previous sections.
We divide cases according to its phase.

\textit{Case $Q_i$ is the $p$-intrinsic}.
We have from Lemma~\ref{p_stopping} that
\begin{align*}
    \begin{split}
    \la_i^p|Q_i|
    &=\iint_{Q_i\cap \Psi(\kappa\La,r_2)^c} H(z,|\na u|)\,dz+\iint_{Q_i\cap \Psi(\kappa\La,r_2)} H(z,|\na u|)\,dz\\
    &\qquad+\iint_{Q_i\cap \Psi(\kappa\delta\La,r_2)^c}\delta^{-1}H(z,|F|) \,dz+\iint_{Q_i\cap \Psi(\kappa\delta\La,r_2)}\delta^{-1}H(z,|F|) \,dz.
    \end{split}
\end{align*}
To proceed further, we note that $\La=\la_i^p+a(z_i)\la_i^q\le (K^2+1)\la_i^p$ and thus
\[
\iint_{Q_i\cap \Psi(\kappa\La,r_2)^c} H(z,|\na u|)\,dz\le \iint_{Q_i\cap \Psi(\kappa\La,r_2)^c} \kappa\La\,dz\le \frac{1}{4}\la_i^p|Q_i|.
\]
Similarly we also have
\[
\iint_{Q_i\cap \Psi(\kappa\delta\La,r_2)^c} H(z,|F|)\,dz\le \frac{1}{4}\la_i^p|Q_i|.
\]
Therefore we deduce from the stopping time argument that
\begin{align}\label{p_cz_decom}
    |Q_i|
    \le\frac{2}{\la_i^p}\iint_{Q_i\cap \Psi(\kappa\La,r_2)} H(z,|\na u|)\,dz+\frac{2}{\la_i^p}\iint_{Q_i\cap \Psi(\kappa\delta\La,r_2)}\delta^{-1}H(z,|F|) \,dz.
\end{align}

On the other hand, by Lemma~\ref{p_lip} and Corollary~\ref{p_com}, there exists a map $v_i\in L^\infty(VI_i; W^{1,\infty}(VB_i,\mathbb{R}^N))$ such that
\[
\iint_{VQ_i}H(z,|\na u-\na v_i|)\,dz\le \epsilon\la_i^p|Q_i|,\qquad \| \na v_i \|_{L^\infty(VQ_i)}\le \left(  \frac{S_\delta}{2^{q+3}}  \right)^\frac{1}{q}\la_i,
\]
where $B_i$ and $I_i$ are projections of $Q_i$ on the spatial direction and the time directions respectively and $S_\delta=S(\data,\delta)>2^{q+3}$ is a constant. Since $[a]_\alpha (Vl_i)^\alpha \la_i^q\le \la_i^p$ where $l_i$ is the radius of $Q_i$, we obtain that for a.e. $z\in VQ_i$,
\[
H(z,|\na v_i|)\le H(z_i,|\na v_i|)+[a]_\alpha(Vl_i)^\alpha \le \frac{S_\delta}{2^{q+2}}\La.
\]
Furthermore, the following estimate can be derived from the above display.
\begin{align}\label{Bogelein_trick}
    H(z,|\na v_i(z)|)\le H(z,|\na u(z)-\na v_i(z)|)\quad\text{for a.e.}\quad z\in VQ_i\cap \Psi(S_\delta\La,r_1).
\end{align}
Indeed, if \eqref{Bogelein_trick} is false, then there exists a point $\om$ in the reference set that $H(\om,|\na v_i(\om)|)> H(\om,|\na u(\om)-\na v_i(\om)|)$ and this leads
\begin{align*}
    \begin{split}
        H(\om,|\na v_i(\om)|)
        &\le \frac{S_\delta}{2^{q+2}}\La\\
        &\le \frac{1}{2^{q+2}}H(\om,|\na u(\om)|)\\
        &\le \frac{2^q}{2^{q+2}}(H(\om,|\na u(\om)-\na v_i(\om)|)+H(\om,|\na v_i(\om)|) )\\
        &\le \frac{1}{2}H(\om,|\na v_i(\om)|).
    \end{split}
\end{align*}
As the above inequality means
\[
0=H(\om,|\na v_i(\om)|)> H(\om,|\na u(\om)-\na v_i(\om)|)=H(\om,|\na u(\om)|)>S_\delta\La,
\]
we get the contradiction and \eqref{Bogelein_trick} holds true. It follows that
\begin{align*}
    \begin{split}
        &\iint_{VQ_i\cap \Psi(S_\delta \La,r_1)} H(z,|\na u|)\,dz\\
        &\le 2^q\iint_{VQ_i\cap \Psi(S_\delta \La,r_1)} (H(z,|\na u-\na v_i|) +H(z,|\na v_i|) )\,dz\\
        &\le 2^{q+1}\iint_{VQ_i\cap \Psi(S_\delta \La,r_1)} (H(z,|\na u-\na v_i|)\,dz\\
        &\le 2^{q+1}\epsilon\la_i^p|Q_i|.
    \end{split}
\end{align*}
    Inserting the above \eqref{p_cz_decom} to the right hand side of the above inequality, we obtain
    \begin{align}\label{main_proof}
        \begin{split}
            \iint_{VQ_i\cap \Psi(S_\delta \La,r_1)} H(z,|\na u|)\,dz
            &\le 2^{q+2}\epsilon \iint_{Q_i\cap \Psi(\kappa\La,r_2)} H(z,|\na u|)\,dz\\
            &\qquad+2^{q+2}\iint_{Q_i\cap \Psi(\kappa\delta\La,r_2)}\delta^{-1}H(z,|F|) \,dz.
        \end{split}
    \end{align}

\textit{Case $Q_i$ is the $(p,q)$-intrinsic}. The argument to obtain \eqref{main_proof} is analogous to the previous case as it is enough to replace used lemmas in $p$-intrinsic case by corresponding lemmas in $(p,q)$-intrinsic case instead. We omit the details.

As for each $i\in\mathbb{N}$, \eqref{main_proof} holds, we use the pairwise disjointedness of $Q_i$ to have
\begin{align*}
    \begin{split}
        \iint_{\Psi(S_\delta \La,r_1)} H(z,|\na u|)\,dz
        &\le \sum_{i\in\mathbb{N}}\iint_{VQ_i\cap \Psi(S_\delta \La,r_1)} H(z,|\na u|)\,dz\\
            &\le 2^{q+2}\epsilon \iint_{\Psi(\kappa\La,r_2)} H(z,|\na u|)\,dz\\
            &\qquad+2^{q+2}\iint_{\Psi(\kappa\delta\La,r_2)}\delta^{-1}H(z,|F|) \,dz.
    \end{split}
\end{align*}
Following the standard Fubini argument in \cite{K2024}, we have
\begin{align*}
    \begin{split}
        &\iint_{Q_{r_1}(z_0)} H(z,|\na u|)(H(z,|\na u|)_k)^{\sigma-1}\,dz\\
         &\le 2^{q+2}\epsilon \iint_{Q_{r_2}(z_0)} H(z,|\na u|)(H(z,|\na u|)_k)^{\sigma-1}\,dz\\
            &\qquad+2\left( \frac{32V\rho}{r_2-r_1}  \right)^{\frac{2q(n+2)(\sigma-1)}{p(n+2)-2n}}(S_\delta\La_0)^{\sigma-1}\iint_{Q_{2\rho}(z_0)}   H(z,|\na u|)   \,dz\\
            &\qquad+2^{q+2}\iint_{Q_{2\rho}(z_0)}\delta^{-1}H(z,|F|) \,dz,
    \end{split}
\end{align*}
where we denoted
\[
H(z,|\na u(z)|)_k=\min\{ H(z,|\na u(z)|),k \}
\]
for some $k>0$. By taking
\begin{align}\label{def_ep}
    \epsilon=\frac{1}{2^{q+3}},
\end{align}
and applying Lemma~\ref{iter_lem}, we obtain
\begin{align*}
    \begin{split}
        &\iint_{Q_{\rho}(z_0)} H(z,|\na u|)(H(z,|\na u|)_k)^{\sigma-1}\,dz\\
         &\le c\La_0^{\sigma-1}\iint_{Q_{2\rho}(z_0)} H(z,|\na u|)\,dz+c\iint_{Q_{2\rho}(z_0)}H(z,|F|) \,dz,
    \end{split}
\end{align*}
    where $c=c(\data,\sigma)$. The conclusion follows by letting $k$ to infinity and substituting $\La_0$ into the above inequality.
\end{proof}

\textbf{Acknowledgement}

W. Kim was supported by the Wallenberg AI, Autonomous Systems and Software Program (WASP) funded by the Knut and Alice Wallenberg Foundation.

%\section*{References}


\begin{thebibliography}{10}

\bibitem{MR2286632}
E.~Acerbi and G.~Mingione.
\newblock Gradient estimates for a class of parabolic systems.
\newblock {\em Duke Math. J.}, 136(2):285--320, 2007.

\bibitem{baasandorj2023self}
S.~Baasandorj, S-S. Byun, and W.~Kim.
\newblock Self-improving properties of very weak solutions to double phase
  systems.
\newblock {\em Trans. Am. Math. Soc.},
  376(12):8733--8768, 2023.

\bibitem{MR3348922}
P.~Baroni, M.~Colombo, and G.~Mingione.
\newblock Harnack inequalities for double phase functionals.
\newblock {\em Nonlinear Anal.}, 121:206--222, 2015.

\bibitem{BARONI2017593}
P.~Baroni and C.~Lindfors.
\newblock The {C}auchy–{D}irichlet problem for a general class of parabolic
  equations.
\newblock {\em Ann. Inst. H. Poincaré Anal. Non Linéaire}, 34(3):593--624,
  2017.

\bibitem{MR3073153}
V.~B\"{o}gelein, F.~Duzaar, and P.~Marcellini.
\newblock Parabolic systems with {$p,q$}-growth: a variational approach.
\newblock {\em Arch. Ration. Mech. Anal.}, 210(1):219--267, 2013.

\bibitem{MR3035434}
S.~Byun, J.~Ok, and S.~Ryu.
\newblock Global gradient estimates for general nonlinear parabolic equations
  in nonsmooth domains.
\newblock {\em J. Differential Equations}, 254(11):4290--4326, 2013.

\bibitem{MR2187159}
S.~Byun and L.~Wang.
\newblock Parabolic equations in {R}eifenberg domains.
\newblock {\em Arch. Ration. Mech. Anal.}, 176(2):271--301, 2005.

\bibitem{MR2345911}
S.~Byun, L.~Wang, and S.~Zhou.
\newblock Nonlinear elliptic equations with {BMO} coefficients in {R}eifenberg
  domains.
\newblock {\em J. Funct. Anal.}, 250(1):167--196, 2007.

\bibitem{MR1486629}
L.~A. Caffarelli and I.~Peral.
\newblock On {$W^{1,p}$} estimates for elliptic equations in divergence form.
\newblock {\em Comm. Pure Appl. Math.}, 51(1):1--21, 1998.

\bibitem{MR3360738}
M.~Colombo and G.~Mingione.
\newblock Bounded minimisers of double phase variational integrals.
\newblock {\em Arch. Ration. Mech. Anal.}, 218(1):219--273, 2015.

\bibitem{MR3294408}
M.~Colombo and G.~Mingione.
\newblock Regularity for double phase variational problems.
\newblock {\em Arch. Ration. Mech. Anal.}, 215(2):443--496, 2015.

\bibitem{MR3447716}
M.~Colombo and G.~Mingione.
\newblock Calder\'{o}n-{Z}ygmund estimates and non-uniformly elliptic
  operators.
\newblock {\em J. Funct. Anal.}, 270(4):1416--1478, 2016.

\bibitem{cupini2024regularity}
G.~Cupini, P.~Marcellini, and E.~Mascolo.
\newblock Regularity for nonuniformly elliptic equations with p, q-growth and
  explicit x, u-dependence.
\newblock {\em Arch. Ration. Mech. Anal.}, 248(4):60, 2024.

\bibitem{MR3985927}
C.~De~Filippis and G.~Mingione.
\newblock A borderline case of {C}alder\'{o}n-{Z}ygmund estimates for
  nonuniformly elliptic problems.
\newblock {\em St. Petersburg Math. J.}, 31(3):455--477, 2020.

\bibitem{de2023regularity}
C.~De~Filippis and G.~Mingione.
\newblock Regularity for double phase problems at nearly linear growth.
\newblock {\em Arch. Ration. Mech. Anal.}, 247(5):85, 2023.

\bibitem{de2024regularity}
F.~{De Filippis} and M.~Piccinini.
\newblock Regularity for multi-phase problems at nearly linear growth.
\newblock {\em J. Differential Equations}, 410:832--868, 2024.

\bibitem{MR1230384}
E.~DiBenedetto.
\newblock {\em Degenerate parabolic equations}.
\newblock Universitext. Springer-Verlag, New York, 1993.

\bibitem{MR1246185}
E.~DiBenedetto and J.~Manfredi.
\newblock On the higher integrability of the gradient of weak solutions of
  certain degenerate elliptic systems.
\newblock {\em Amer. J. Math.}, 115(5):1107--1134, 1993.

\bibitem{MR2076158}
L.~Esposito, F.~Leonetti, and G.~Mingione.
\newblock Sharp regularity for functionals with {$(p,q)$} growth.
\newblock {\em J. Differential Equations}, 204(1):5--55, 2004.

\bibitem{MR2058167}
I.~Fonseca, J.~Mal\'{y}, and G.~Mingione.
\newblock Scalar minimizers with fractal singular sets.
\newblock {\em Arch. Ration. Mech. Anal.}, 172(2):295--307, 2004.

\bibitem{MR1962933}
E.~Giusti.
\newblock {\em Direct methods in the calculus of variations}.
\newblock World Scientific Publishing Co., Inc., River Edge, NJ, 2003.

\bibitem{MR722254}
T.~Iwaniec.
\newblock Projections onto gradient fields and {$L^{p}$}-estimates for
  degenerated elliptic operators.
\newblock {\em Studia Math.}, 75(3):293--312, 1983.

\bibitem{K2024}
W.~Kim.
\newblock Calder\'on-zygmund type estimate for the parabolic double-phase
  system.
\newblock {\em to appear in Ann. Sc. Norm. Super. Pisa Cl. Sci.}, 2025.

\bibitem{KKM}
W.~Kim, J.~Kinnunen, and K.~Moring.
\newblock Gradient higher integrability for degenerate parabolic double-phase
  systems.
\newblock {\em Arch. Ration. Mech. Anal.}, 247(5):Paper No. 79, 46, 2023.

\bibitem{KKS}
W.~Kim, J.~Kinnunen, and L.~Särkiö.
\newblock Lipschitz truncation method for parabolic double-phase systems and
  applications.
\newblock {\em J. Funct. Anal.}, 288(3):110738, 2025.

\bibitem{kim2024holder}
W.~Kim, K.~Moring, and L.~Särkiö.
\newblock H\"older regularity for degenerate parabolic double-phase equations.
\newblock {\em arXiv}, 2024.

\bibitem{MR4718687}
W.~Kim and L.~S\"arki\"o.
\newblock Gradient higher integrability for singular parabolic double-phase
  systems.
\newblock {\em NoDEA Nonlinear Differential Equations Appl.}, 31(3):Paper No.
  40, 38, 2024.

\bibitem{MR1749438}
J.~Kinnunen and J.~L. Lewis.
\newblock Higher integrability for parabolic systems of {$p$}-{L}aplacian type.
\newblock {\em Duke Math. J.}, 102(2):253--271, 2000.

\bibitem{MR1720770}
J.~Kinnunen and S.~Zhou.
\newblock A local estimate for nonlinear equations with discontinuous
  coefficients.
\newblock {\em Comm. Partial Differential Equations}, 24(11-12):2043--2068,
  1999.

\bibitem{MR969900}
P.~Marcellini.
\newblock Regularity of minimizers of integrals of the calculus of variations
  with nonstandard growth conditions.
\newblock {\em Arch. Rational Mech. Anal.}, 105(3):267--284, 1989.

\bibitem{MR1094446}
P.~Marcellini.
\newblock Regularity and existence of solutions of elliptic equations with
  {$p,q$}-growth conditions.
\newblock {\em J. Differential Equations}, 90(1):1--30, 1991.

\bibitem{marcellini2023local}
P.~Marcellini.
\newblock Local lipschitz continuity for {$p,q$}- pdes with explicit $u$-
  dependence.
\newblock {\em Nonlinear Analysis}, 226:113066, 2023.

\bibitem{MR3356846}
T.~Singer.
\newblock Parabolic equations with {$p,q$}-growth: the subquadratic case.
\newblock {\em Q. J. Math.}, 66(2):707--742, 2015.

\bibitem{MR3532237}
T.~Singer.
\newblock Existence of weak solutions of parabolic systems with {$p,
  q$}-growth.
\newblock {\em Manuscripta Math.}, 151(1-2):87--112, 2016.

\end{thebibliography}
\end{document}